\documentclass[12pt]{amsart}

\usepackage[margin=1in]{geometry}
\usepackage{amsmath,amsthm,amssymb}
\usepackage{xcolor}
\usepackage{soul}
\usepackage{mathrsfs}
\usepackage{enumerate}
\usepackage{bbm}

\usepackage[unicode]{hyperref}
\hypersetup{colorlinks=true, citecolor=blue, linkcolor=blue, urlcolor=blue, pdfstartview=FitH, pdfauthor=Jesse Thorner and Zhuo Zhang, pdftitle=A uniform Chebotarev density theorem with Artin's holomorphy conjecture}

\usepackage[capitalise]{cleveref}

\newtheorem{theorem}{Theorem}
\newtheorem*{theorem*}{Theorem}
\newtheorem{lemma}[theorem]{Lemma}
\newtheorem{proposition}[theorem]{Proposition}

\theoremstyle{remark} \newtheorem*{remark}{Remark}

\newtheoremstyle{named}{}{}{\itshape}{}{\bfseries}{.}{.5em}{\thmnote{#3}}
\theoremstyle{named}

\usepackage[justification=centering]{caption}

\numberwithin{equation}{section}
\numberwithin{theorem}{section}

\usepackage{stackengine}

\newcommand{\kp}{\mathfrak{p}}
\newcommand{\ka}{\mathfrak{a}}
\newcommand{\kd}{\mathfrak{d}}
\newcommand{\kP}{\mathfrak{P}}
\newcommand{\Li}{\operatorname{Li}}
\newcommand{\kf}{\mathfrak{f}}
\newcommand{\f}{\mathfrak{f}}

\newcommand{\Gal}{\mathrm{Gal}}
\newcommand{\re}{\mathrm{Re}}
\renewcommand{\Re}{\operatorname{Re}}

\newcommand{\im}{\mathrm{Im}}

\newcommand{\Q}{\mathbb{Q}}
\newcommand{\cQ}{\mathcal{Q}}

\newcommand{\cD}{\mathcal{D}}

\newcommand{\cK}{\mathcal{K}}
\newcommand{\cA}{\mathcal{A}}

\newcommand{\R}{\mathbb{R}}

\newcommand{\N}{\mathrm{N}}

\newcommand{\Norm}{\mathrm{N}}
\newcommand{\Z}{\mathbb{Z}}
\newcommand{\reg}{\mathrm{reg}}
\newcommand{\Ind}{\operatorname{Ind}}
\newcommand{\cO}{\mathcal{O}}

\newcommand{\Irr}{\mathrm{Irr}}
\newcommand{\Rep}{\mathrm{Rep}}
\newcommand{\Log}{\operatorname{Log}}

\renewcommand{\bar}{\overline}
\renewcommand{\epsilon}{\varepsilon}
\renewcommand{\hat}{\widehat}
\renewcommand{\tilde}{\widetilde}

\newcommand{\OK}{\mathcal{O}_K}
\newcommand{\OL}{\mathcal{O}_L}
\newcommand{\AHC}{\operatorname{AHC}}
\newcommand{\Ab}{\operatorname{Ab}}

\usepackage{cancel}

\usepackage{constants}
\usepackage{bbm}
\newconstantfamily{abcon}{symbol=c}

\newconstantfamily{B}{symbol=B}

\newconstantfamily{A}{symbol=\alpha}

\title[A new bound on the relative error in the Chebotarev density theorem]{A new bound on the relative error in the Chebotarev density theorem}

\author{Jesse Thorner}
\author{Zhuo Zhang}
\address{Department of Mathematics, University of Illinois Urbana-Champaign, Urbana, IL 61801}
\email{jesse.thorner@gmail.com}
\email{zhuoz4@illinois.edu}

\begin{document}

\begin{abstract}
We unconditionally improve the uniformity in the Chebotarev density theorem for Galois extensions of number fields using nonabelian base change.  This leads to the first theoretical improvement over Weiss's bound for the least norm of an unramified prime ideal with given Artin symbol.
\end{abstract}

\maketitle

\section{Introduction and statement of the main results}
\label{sec:intro}

Throughout this paper, the numbers $c_1,c_2,c_3,\ldots$ form a certain sequence of positive, absolute, and effectively computable constants.

\subsection{Introduction}

For a number field $F$, let $\mathcal{O}_F$ be the ring of integers, $D_F$ the absolute value of the discriminant, and $n_F=[F:\Q]$. Let $L/F$ be a Galois extension of number fields, and let $G=\Gal(L/F)$.  Let $\kp$ (resp. $\ka$) denote a nonzero prime ideal (resp. nonzero ideal) of $\cO_F$.  Define $\N_{F/\Q}\ka=|\cO_F/\ka|$.  For a prime ideal $\kp$ of $\cO_F$ that is unramified in $L$, the Artin symbol $[\frac{L/F}{\kp}]$ denotes the conjugacy class in $G$ of Frobenius automorphisms attached to the prime ideals $\kP$ of $\cO_L$ lying over $\kp$. (See \cref{subsec:prelim}.) For a conjugacy class $C\subseteq G$, we define
\begin{align*}
\mathbb{P}_C(L/F)&=\{\kp\subseteq\cO_F\colon \textup{$\kp$ unramified in $L$ and $[\tfrac{L/F}{\kp}]=C$}\},\\
\pi_C(x,L/F)&=\#\{\kp\in\mathbb{P}_C(L/F)\colon \N_{F/\Q}\kp\leq x\}.
\end{align*}
The Chebotarev density theorem states that
\begin{equation}
\label{eqn:CDT}
\lim_{x\to\infty}\frac{\pi_C(x,L/F)}{\Li(x)}=\frac{|C|}{|G|},\qquad\Li(x) = \int_2^x \frac{dt}{\log t}.
\end{equation}
The rate of convergence lies at the center of many deep problems in number theory.  Assuming the generalized Riemann hypothesis, Lagarias and Odlyzko \cite[Theorem 1.1]{LO} established a very strong rate of convergence.  In this paper, we will focus on unconditional bounds for the rate of convergence.

To describe the best known result, we need some notation.  Let $\Irr(G)$ be the set of characters associated to the (complex) irreducible representations of $G$.  When $G$ is abelian, $\Irr(G)$ is a group that is isomorphic to $G$, and the characters $\chi\in\Irr(G)$ are Hecke characters.  We write $\mathbbm{1}=\mathbbm{1}_G$ for the trivial character of $G$.

Given $\chi\in\Irr(G)$, let $L(s,\chi,L/F)$ be the associated Artin $L$-function.  We say that a group $G$ satisfies the Artin holomorphy conjecture (AHC) if for all Galois extensions $L/F$ with Galois group isomorphic to $G$, the Artin $L$-functions $L(s,\chi,L/F)$ associated to $\chi\in\Irr(G)-\{\mathbbm{1}_G\}$ are entire.  Define
\begin{align*}
\AHC(G)&=\{\textup{$H\subseteq G$: $H$ is a subgroup such that $H$ satisfies AHC}\},\\
\Ab(G)&=\{\textup{$H\subseteq G$: $H$ is a subgroup such that $H$ is abelian}\}.
\end{align*}
It follows from class field theory, Artin reciprocity, and work of Hecke that $\mathrm{Ab}(G)\subseteq\mathrm{AHC}(G)$.

Let $\kf_{\chi}$ be its Artin conductor, which is an ideal of $\cO_F$ (see \eqref{eqn:def of Artin conductor}). Write $\chi(1)$ for the dimension of $\chi$. Define
\begin{equation}
\label{eqn:degree_q_def}
    \begin{gathered}
    d_{G} = \max_{\chi\in\Irr(G)}\chi(1),\qquad q_{L/F}=D_F^{d_G^2}\big(\max_{\chi\in\Irr(G)}\N_{F/\Q}\kf_{\chi}\big)^{2d_G},\qquad \Log x=\max\{1,\log x\},\\
\cQ_{L/F} = (d_G^2 n_F)^{d_G^2 n_F\Log\Log d_G}|\Irr(G)|q_{L/F}.
\end{gathered}
\end{equation}
Let $\zeta_L(s)$ be the Dedekind zeta function of $L$, which factors as
\[
\zeta_L(s) = \zeta_F(s)\prod_{\chi\in\Irr(G)-\{\mathbbm{1}_G\}}L(s,\chi,L/F)^{\chi(1)},\qquad \re(s)>1.
\]
Define
\begin{equation}
\label{eqn:def of beta_1}
\beta_1 = \beta_{1,L} = \sup\{\sigma<1\colon \zeta_L(\sigma)=0\}.
\end{equation}
It is classical that $\beta_1<1$.  Stark \cite[Proof~of~Theorem~3]{Stark} proved that if $\beta_1$ is a simple zero, then each $L(s,\chi,L/F)$ associated to $\chi\in\Irr(G)$ is holomorphic at $s=\beta_1$, and 
\begin{equation}
\label{eqn:chi_0_def}
\begin{gathered}
\textup{there exists a unique $\chi_0=\chi_{0,L/F}\in\Irr(G)$ such that $\chi_0=\bar{\chi}_0$, $\chi_0(1)=1$, and}\\
\textup{the Artin $L$-function $L(s,\chi_0,L/F)$ vanishes at $s=\beta_1$ (see \cref{eqn:chi_1 well defined} below).}
\end{gathered}
\end{equation}
Define
\begin{equation}
\label{eqn:def of chi_1}
	\chi_1=\chi_{1,L/F}=\begin{cases}
	\chi_{0}&\mbox{if $\beta_1$ is simple,}\\
\mathbbm{1}_G&\mbox{otherwise.}
\end{cases}
\end{equation}
We always have $\chi_1(1)=1$ and $\chi_1=\bar{\chi}_1$.

\subsection{Main results}

Define the {\it relative error} $\delta_C(x,L/F)$ by
\begin{equation}
\label{eqn:rel_err_1}
\pi_C(x,L/F)=\frac{|C|}{|G|}(\Li(x)-\chi_{1}(C)\Li(x^{\beta_{1}}))(1+\delta_C(x,L/F)).
\end{equation}
Since $\Li(x)>\chi_{1}(C)\Li(x^{\beta_{1}})$, it follows that if $|\delta_C(x,L/F)|<1$, then
\[
1\leq \pi_C(x,L/F)< 4\frac{|C|}{|G|}\Li(x).
\]
The following unconditional result is proved in \cite[Theorem 1.4]{TZ3}.

\begin{theorem}[Thorner--Zaman]
\label{thm:TZ}
Let $L/F$ be a Galois extension of number fields with Galois group $G$, and let $C\subseteq G$ be a conjugacy class.  There exist constants $\Cl[abcon]{thm:main2long_TZ}$, $\Cl[abcon]{thm:main3long_TZ}$, and $\Cl[abcon]{thm:main1long_TZ}$ such that if $H\in \mathrm{Ab}(G)$, $C\cap H$ is nonempty, and $K=L^H$ is the subfield of $L$ fixed by $H$, then
\[
|\delta_C(x,L/F)|\leq \Cr{thm:main2long_TZ}\exp\Big(-\Cr{thm:main3long_TZ}\frac{\log x}{\log q_{L/K}+\sqrt{n_{K}\log x}}\Big)<1,\qquad x\geq (2\cQ_{L/K})^{\Cr{thm:main1long_TZ}}.
\]
\end{theorem}
\begin{remark}
If $H\in\Ab(G)$ and $K=L^H$, then $d_H=1$ and $\cQ_{L/K}=n_K^{n_K}|\Irr(H)| q_{L/K}$.  Per \cite[Lemma 4.1]{TZ2}, there exists a constant $\Cl[abcon]{CFT}>0$ such that $|\Irr(H)|\leq (2 n_K^{n_K} q_{L/K})^{\Cr{CFT}}$.
\end{remark}

Our main result improves upon \cref{thm:TZ} by replacing the condition $H\in\Ab(G)$ with the more flexible condition $H\in\AHC(G)$.
\begin{theorem}
\label{thm:main_long}
Let $L/F$ be a Galois extension of number fields with Galois group $G$, and let $C\subseteq G$ be a conjugacy class.  There exist constants $\Cl[abcon]{thm:main2long}$, $\Cl[abcon]{thm:main3long}$, and $\Cl[abcon]{thm:main1long}$ such that if $H\in \mathrm{AHC}(G)$, $C\cap H$ is nonempty, $K=L^H$ is the subfield of $L$ fixed by $H$, then
\[
|\delta_C(x,L/F)|\leq \Cr{thm:main2long}\exp\Big(-\Cr{thm:main3long}\frac{\log x}{\log q_{L/K}+\sqrt{d_H^2 n_K\log x}}\Big)<1,\qquad x\geq (2\cQ_{L/K})^{\Cr{thm:main1long}\Log d_H}.
\]
\end{theorem}

The range of $x$ in \cref{thm:TZ} implies that
\[
\min\{\textup{$\kp\in \mathbb{P}_C(L/F)$: $\N_{F/\Q}\kp$ prime}\} \leq\min_{\substack{H\in\Ab(G) \\ C\cap H\neq\emptyset}}(2\cQ_{L/L^H})^{\Cr{thm:main1long_TZ}}.
\]
This was first proved by Weiss \cite{Weiss}.  \cref{thm:main_long} immediately yields the first theoretical improvement, replacing the minimum over $H\in\Ab(G)$ with a minimum over $H\in\AHC(G)$.
\begin{theorem}
\label{thm:Linnik}
Let $L/F$ be a Galois extension of number fields with Galois group $G$, and let $C\subseteq G$ be a conjugacy class.  There exists $\kp\in \mathbb{P}_C(L/F)$ of prime norm satisfying
\[
\N_{F/\Q}\kp \leq \min_{\substack{H\in \mathrm{AHC}(G) \\ C\cap H\neq\emptyset}}(2\cQ_{L/L^H})^{\Cr{thm:main1long}\Log d_H}.
\]
\end{theorem}

\begin{remark}
It follows from \cite[Lemma 4.2]{FiorilliJouve} that if $H$ is any subgroup of $G$, then
\begin{equation}
\label{eqn:dL_bound}
\cQ_{L/L^H}\ll_{d_H,n_{L^H}} D_L^{4 d_H^2/|H|}|\Irr(H)|.
\end{equation}
For example, let $\mathbb{N}=\{j\in\Z\colon j\geq 1\}$, $a,q\in\mathbb{N}$, and $\gcd(a,q)=1$.  If $F=\Q$ and $L=\Q(e^{2\pi i/q})$ (so $G=H=(\Z/q\Z)^{\times}$ and $L^H=\Q$), then \cref{thm:Linnik} and \eqref{eqn:dL_bound} recover Linnik's theorem that there exists a prime $p\ll q^{5\Cr{thm:main1long}}$ such that $p\equiv a\pmod{q}$.
\end{remark}

\subsection{Reduction to groups satisfying AHC}

Using the invariance of Artin $L$-functions under induction (see \cref{lem:MMS}), we reduce our proof of \cref{thm:main_long,thm:Linnik} to a study of $\pi_C(x,L/K)$ when $L/K$ is a (possibly non-abelian) Galois extension whose Galois group $G$ satisfies AHC.  The only earlier work estimating $\pi_C(x,L/K)$ under AHC when $L/K$ is non-abelian (i.e., without abelian base change) is that of V.K. Murty \cite[Theorem 4.1]{VKM}, which we now describe using our notation:  Under AHC, there exist constants $\Cl[abcon]{VKM1}$, $\Cl[abcon]{VKM2}$, and $\Cl[abcon]{VKM3}$ such that if 
 \begin{equation}
\label{eqn:VKM_range2}
x\geq (q_{L/K}^{d_G^2})^{\Cr{VKM1}\log(|G|\log (q_{L/K}^{d_G^2}))}+(1+\chi_1(C))e^{\frac{1}{1-\beta_1}},
\end{equation}
then
\begin{equation}
\label{eqn:VKM1}
|\delta_C(x,L/K)|\leq \frac{\Cr{VKM2}}{\sqrt{|C|}}\cdot \frac{n_L (\log x)^3}{1-\chi_1(C)x^{\beta_1-1}}\exp\Big(-\Cr{VKM3}\frac{\log x}{\log(q_{L/K}^{d_G^2})+\sqrt{d_{G}^3 n_K\log x}}\Big)<1.
\end{equation}
Here, we significantly improve upon \eqref{eqn:VKM_range2} and \eqref{eqn:VKM1}.

\begin{proposition}
\label{prop:main_AC_long}
Let $L/K$ be a Galois extension of number fields with Galois group $G$ satisfying AHC, and let $C\subseteq G$ be a conjugacy class.  Recall \eqref{eqn:degree_q_def}-\eqref{eqn:rel_err_1}.  
There exist constants $\Cl[abcon]{cor:main3}$, $\Cl[abcon]{cor:main2}$, and $\Cl[abcon]{cor:main1}$ such that
\[
|\delta_C(x,L/K)|\leq \Cr{cor:main3}\exp\Big(-\Cr{cor:main2}\frac{\log x}{\log q_{L/K}+\sqrt{d_G^2 n_K \log x}}\Big)<1,\qquad x\geq (2\cQ_{L/K})^{\Cr{cor:main1}\Log d_G}.
\]
\end{proposition}


\cref{prop:main_AC_long} improves upon the results of \cite{VKM} (leading to the unconditional \cref{thm:main_long,thm:Linnik}) in four crucial ways.  First, the dependence on $\beta_1$ and $\chi_1$ in \eqref{eqn:VKM_range2} and \eqref{eqn:VKM1} is altogether removed in \cref{prop:main_AC_long}.  Second, the factor of $(\log x)^3$ in the bound for $|\delta_C(x,L/K)|$ in \eqref{eqn:VKM1} is removed.  Third, the range of $x$ in \eqref{eqn:VKM_range2} is improved from a super-polynomial dependence on $q_{L/K}$ to a polynomial dependence.  Fourth, compared to \eqref{eqn:VKM_range2} and \eqref{eqn:VKM1}, \cref{prop:main_AC_long} essentially replaces $q_{L/K}^{d_G^2}$ with $q_{L/K}$.  This last feature is important because $d_G^2$ can be as large as $|G|$, which could be large with respect to $q_{L/K}$.

These improvements are made possible by a new zero-free region (\cref{thm:ZFR}) and a new ``repulsive'' log-free zero density estimate that improves when $\beta_1$ is simple and close to $s=1$ (\cref{thm:LFZDE}) for the Artin $L$-functions associated to $\chi\in\Irr(G)$.  If $d_G=1$, then AHC is known, and these results are already available in \cite{Weiss}.  If $d_G\geq 2$ and we assume a strong form of AHC in which each $L$-function $L(s,\chi,L/K)$ is in fact the $L$-function of a cuspidal automorphic representation of $\mathrm{GL}_{\chi(1)}(\mathbb{A}_K)$, then we would be able to cite existing results in the literature, though the dependence on $d_G$ and $n_K$ would be either unspecified or far from optimal.  Therefore, we must establish these results using {\it only} AHC while maintaining strong uniformity in $d_G$ and $n_K$ (especially in the exponent $q_{L/K}$).  This requires more care, including several auxiliary results on Artin $L$-functions that we could not find elsewhere in the literature, which might be of independent interest.

While our new zero-free region in \cref{thm:ZFR} is a nontrivial refinement of the work in \cite{VKM}, our log-free zero density estimate in \cref{thm:LFZDE} appears to be genuinely new, especially the hybrid-aspect ``pre-sifted'' large sieve inequality in \cref{prop:hybrid_MVT}.  Without significant improvements to both the zero-free regions and the bounds on the critical line that we prove for Artin $L$-functions in this paper, our improvement from $q_{L/K}^{d_G^2}$ in \eqref{eqn:VKM_range2} and \eqref{eqn:VKM1} to $q_{L/K}$ in \cref{prop:main_AC_long} is the natural limit of our methods.

The above discussion pertains to the AHC-conditional \cref{prop:main_AC_long}, but the true payoff lies in unconditional \cref{thm:main_long,thm:Linnik}.  These constitute theoretical improvements over the work of Thorner--Zaman \cite[Theorem~1.4]{TZ3}, which already improves upon the unconditional effective forms of the Chebotarev density theorem established by Lagarias--Odlyzko \cite[Theorem~1.3]{LO}, Lagarias--Montgomery--Odlyzko \cite[Theorem~1.1]{LO}, Serre \cite{Serre}, and Thorner--Zaman (\cite[Theorem~1.1]{TZ1} and \cite[Theorem~1.1]{TZ2}).

\subsection*{Organization}

In \cref{sec:comparison}, we compare \cref{thm:TZ,thm:main_long}.  In \cref{sec:basics,sec:AHC}, we collect results about Artin $L$-functions that arise in our proofs.  We also state and prove our new zero-free region (\cref{thm:ZFR}) and state our new log-free zero density estimate (\cref{thm:LFZDE}) for Artin $L$-functions satisfying AHC.  In \cref{sec:reduction}, we prove \cref{thm:main_long} assuming \cref{prop:main_AC_long}.  In \cref{sec:proof_main}, we prove \cref{prop:main_AC_long} using \cref{thm:ZFR,thm:LFZDE}.  In \cref{sec:MVT,sec:LFZDE}, we prove \cref{thm:LFZDE}.

\subsection*{Acknowledgements}

The authors thank Casey Appleton, Robert Lemke Oliver, and Asif Zaman for helpful discussions.  JT is partially supported by the National Science Foundation (DMS-2401311) and the Simons Foundation (MP-TSM-00002484).

\section{Comparison with earlier results}
\label{sec:comparison}

Since $\Ab(G)\subseteq\AHC(G)$, \cref{thm:main_long} is at least as strong as \cref{thm:TZ} for all Galois extensions $L/F$.  It is difficult to fully classify the extensions $L/F$ and the conjugacy classes $C$ such that \cref{thm:main_long} greatly improves upon \cref{thm:TZ}.  To simplify our contrast between  \cref{thm:main_long} and \cref{thm:TZ}, we use \eqref{eqn:dL_bound}, leaving us to determine whether
\begin{equation}
\label{eqn:compare5}
\textup{$\min_{\substack{H\in\mathrm{AHC}(G) \\ C\cap H\neq\emptyset}}\frac{d_{H}^2\Log d_{H}}{|H|}$ is much smaller than $\min_{\substack{H\in\mathrm{Ab}(G) \\ C\cap H\neq \emptyset}}\frac{1}{|H|}$.}
\end{equation}
In \eqref{eqn:compare5}, the first quantity is no larger than the second because $\mathrm{Ab}(G)\subseteq\mathrm{AHC}(G)$, and an abelian subgroup $H$ satisfies $d_H=1$.  There might exist a nonabelian $H\in\mathrm{AHC}(G)$ intersecting $C$ such that the first quantity is much smaller than the second.  For example, $\mathrm{AHC}(G)$ contains all monomial subgroups of $G$.  Every nilpotent subgroup of $G$ is monomial; in particular, the Sylow subgroups of $G$ lie in $\mathrm{AHC}(G)$.  Therefore, in \eqref{eqn:compare5}, the condition $H\in\mathrm{AHC}(G)$ is much more flexible than the condition $H\in\mathrm{Ab}(G)$.

Despite the discussion above, for any given group $G$, it is difficult to classify the conjugacy classes $C$ for which \eqref{eqn:compare5} is true.  To simplify the discussion further, we observe that if $H\in\mathrm{Ab}(G)$ and $g\in H\cap C$, then all elements of $H$ commute with $g$ because $H$ is abelian.  It follows that $H$ is contained in the centralizer subgroup $C_G(g)$ of $g$.  The orbit-stabilizer theorem now implies that $[G:H]\geq [G:C_G(g)]=|C|$.  Therefore, we argue that \cref{thm:main_long} improves upon \cref{thm:TZ} when
\begin{equation}
	\label{eqn:compare_unconditional}
	\textup{$\min_{\substack{H\in\mathrm{AHC}(G) \\ C\cap H\neq\emptyset}}\frac{d_{H}^2\Log d_{H}}{|H|}$ is much smaller than $\frac{|C|}{|G|}$.}
\end{equation}
When $G$ happens to satisfy AHC, we argue that \cref{thm:main_long} improves upon \cref{thm:TZ} when
\begin{equation}
\label{eqn:compare6}
\textup{$d_{G}^2\Log d_{G}$ is much smaller than $|C|$.}
\end{equation}
Compared to \eqref{eqn:compare5}, the condition \eqref{eqn:compare6} is purely representation theoretic and more computationally tractable.  We present some infinite families of examples of nonabelian groups $G$ with conjugacy classes $C$ such that \eqref{eqn:compare_unconditional} or \eqref{eqn:compare6} is true.

\subsection*{Example: Dihedral groups}

Let $n\geq 3$ be an integer, and let $G$ be the dihedral group of order $2n$.  The group $G$ is monomial and thus satisfies AHC and $d_G^2 \Log d_G=4$.  Consider the presentation $G = \langle r,s\mid r^n = s^2 = rsrs = 1\rangle$.  If $n$ is odd, then consider the conjugacy class $\{sr^j\colon 0\leq j\leq n-1\}$.  If $n$ is even and $k\in\{0,1\}$, then consider the conjugacy class $\{sr^{2j+k}\colon 0\leq j\leq (n-2)/2\}$.  In any case, \eqref{eqn:compare6} is satisfied when $n$ is sufficiently large.

\subsection*{Example: Groups of order $pq$}

Let $p,q$ be primes such that $p\equiv 1\pmod{q}$.  There exists a unique isomorphism class of nonabelian groups of order $pq$ with presentation
\[
G = \langle a,b\mid a^p=b^q =1,~bab^{-1}= a^r\rangle,
\]
where $r$ is any integer in a residue class of $(\Z/p\Z)^{\times}$ with multiplicative order $q$.  This group can be realized as the semidirect product $\Z/p\Z\rtimes\Z/q\Z$.  The group $G$ is supersolvable with normal series $\{1\}\unlhd \Z/p\Z\unlhd G$.  Thus, $G$ is monomial and satisfies AHC.  For any integer $n\in[1,q-1]$, the set $C = \{a^j b^n\colon 0\leq j\leq p-1\}$ is a conjugacy class in $G$ of size $p$.

If $\chi\in\Irr(G)$, then $\chi(1)$ is a proper divisor of $pq$, so $\chi(1)\in\{1,p,q\}$.  Since $G$ is nonabelian and $d_G^2<|G|=pq<p^2$, it follows that $d_G=q$.  The infinitude of pairs of primes $p,q$ such $p\equiv 1\pmod{q}$ and $p$ is much larger than $q^2\log q$ follows from Dirichlet's theorem.  For such $p$ and $q$, $G$ and $C$ satisfy \eqref{eqn:compare6}.

\subsection*{Example: Symmetric groups}

Consider the symmetric group $G = S_n$.  Per \cite[Theorem 1]{VershikKerov}, there exist constants $0<\Cl[abcon]{dG1}<\Cl[abcon]{dG2}$ such that $e^{-\Cr{dG2}\sqrt{n}}n!\leq d_{G}^2\leq e^{-\Cr{dG1}\sqrt{n}}n!$.  By Stirling's formula, there exist constants $0<\Cl[abcon]{dG3}<\Cl[abcon]{dG4}$ such that
\begin{equation}\label{eqn:Sn irrep dimension}
	e^{-\Cr{dG4}\sqrt{n}}n!\leq d_{G}^2\Log d_{G} \leq e^{-\Cr{dG3}\sqrt{n}}n!.
\end{equation}
Let $k\geq 1$ be an integer, let $n_1 > n_2 > \cdots > n_k\ge 1$ be integers, and let $r_1,\ldots,r_k\geq 1$ be integers.  Consider the partition $\mu$ of $n$ given by
\[
n = \underbrace{n_1+\cdots+n_1}_{\textup{$r_1$ times}}+\underbrace{n_2+\cdots+n_2}_{\textup{$r_2$ times}}+\cdots+\underbrace{n_k+\cdots+n_k}_{\textup{$r_k$ times}}.
\]
If $C$ is the conjugacy class of permutations of cycle type $\mu$, then
\[
|C| = \frac{n!}{w(\mu)},\qquad w(\mu) = \prod_{j=1}^k n_j^{r_j}(r_j!).
\]

\subsubsection{Unconditional example}

Let $p$ be prime, and let $n=p^2$. Let $C$ be the conjugacy class of $p^2$-cycles. Then $|G|/|C|=p^2$. Let $H$ be the Sylow $p$-subgroup of $G$. We have that $H\in\AHC(G)$ since $H$ is a $p$-group and hence is monomial.

Per \cite[Theorem 7.27]{Rotman}, $H$ is isomorphic to the semidirect product $(\Z/p\Z)^p\rtimes(\Z/p\Z)$, where the second factor $\Z/p\Z$ acts on the $p$-fold direct product $(\Z/p\Z)^p$ by cyclically permuting the components. Note that $H$ contains an element of order $p^2$, so $H$ must intersect $C$. We have $|H|=p^{p+1}$. It follows from \cite[Theorem 25]{Serre_linear} that $d_H=p$. Since $(p^2\log p)/p^{p+1}$ is much smaller than $1/p^2$, we see that \eqref{eqn:compare_unconditional} holds.

\subsubsection{AHC-conditional examples}

If $S_n$ satisfies AHC, then \eqref{eqn:compare6} is satisfied for any conjugacy class $C$ such that
\begin{equation}
\label{eqn:conditionSn}
\textup{$w(\mu)$ is is much smaller than $e^{\Cr{dG4}\sqrt{n}}$.}
\end{equation}
If $n$ is suitably large, then there are many conjugacy classes satisfying \eqref{eqn:conditionSn}.  For instance, if $C$ is the conjugacy class of $n$-cycles, then $|C|=(n!)/n=(n-1)!$.  Therefore, if $n$ is sufficiently large, then \eqref{eqn:compare6} is satisfied.  Also, if $m\geq 1$ is an integer, $n=2m+1$, and $\mu$ is the partition $n=(m+1)+m$, then $w(\mu)=m(m+1)=(n^2-1)/4$.  This is much smaller than $e^{\Cr{dG4}\sqrt{n}}$ when $n$ is suitably large with respect to $\Cr{dG4}$.

\begin{remark}
If we assume AHC for $G$ in this example, then by \eqref{eqn:Sn irrep dimension}, we would have
\[
\frac{d_G^2\Log d_G}{|G|}\ge \frac{1}{e^{\Cr{dG4}p}},
\]
which is much larger than $(p^2\log p)/p^{p+1}$. This is a peculiar instance where even if we know AHC for $G$, it is still much more efficient to base change to a subextension. This is due to the presence of the factors $d_H^2\Log d_H$.  Note that $|H|$ is smaller than $|G|$, but $d_H$ is also much smaller than $d_G$. The decrease in the dimensions of the irreducible representations eventually makes it more efficient to base change to $H$.
\end{remark}

\subsection*{Example: Sylow $p$-subgroups}

In general, it can be difficult to compute
\[
\min_{\substack{H\in\mathrm{AHC}(G) \\ C\cap H\neq \emptyset}}\frac{d_H^2\Log d_H}{|H|}.
\]
If $Z(H)$ is the center of $H$, then we always have the bound
\[
\frac{|H/Z(H)|}{|\Irr(H)|}\leq d_H^2 \leq |H/Z(H)|.
\]
It follows that
\[
\min_{\substack{H\in\mathrm{AHC}(G) \\ C\cap H\neq \emptyset}}\frac{d_H^2\Log d_H}{|H|}\leq \min_{\substack{H\in\mathrm{AHC}(G) \\ C\cap H\neq \emptyset}}\frac{\Log|H/Z(H)|}{|Z(H)|}.
\]
This is useful because it is oftentimes more computationally tractable to estimate $|Z(H)|$ than $d_H$.  For example, suppose that there exists a rational prime $p$ such that $C$ intersects a Sylow $p$-subgroup $H$ of $G$.  Let $n\in\mathbb{N}$ satisfy $|H|=p^n$, and let $r\geq 1$ be the rank of $H$.  The lower bound $|Z(H)|\geq p^r$ implies that
\[
\frac{\log|H/Z(H)|}{|Z(H)|}\leq \frac{(n-r)\log p}{p^r}.
\]

\subsection*{New examples from old ones}

Let $G$ (resp. $G'$) be a finite group, and let $C\subseteq G$ (resp. $C'\subseteq G'$) be a conjugacy class.  If $G$ is nonabelian, then so is $G\times G'$.  Note that $C\times C'$ is a conjugacy class in $G\times G'$, and $d_{G\times G'} = d_{G} d_{G'}$.  If  $G\times G'$ satisfies AHC and
\begin{equation}
\label{eqn:compare7}
\textup{$(d_G d_{G'})^2 \Log (d_G d_{G'})$ is much smaller than $|C|\cdot |C'|$,}
\end{equation}
then \eqref{eqn:compare6} is satisfied for the conjugacy class $C\times C'$ in $G\times G'$.

There are many situations where \eqref{eqn:compare7} is satisfied.  For example, if $G$ and $C$ satisfy \eqref{eqn:compare6}, then \eqref{eqn:compare7} is satisfied when $G'$ is abelian, so that $d_{G'}=|C'|=1$.  Alternatively, if both pairs $(G,C)$ and $(G',C')$ satisfy \eqref{eqn:compare6}, then the group $G\times G'$ and the conjugacy class $C\times C'$ satisfy \eqref{eqn:compare7}.  In particular, let $n\geq 2$ be a positive integer and $p_1,\ldots,p_n$ be primes such that $p_1\equiv 1\pmod{p_2}$.  Using the example from our discussion on groups of order $pq$, we can apply these results to the supersolvable (hence monomial) group $(\Z/p_1\Z\rtimes \Z/p_2\Z)\times \Z/p_3\Z\times\cdots\times\Z/p_n \Z$.

\section{Artin $L$-functions}
\label{sec:basics}

In this section, let $L/K$ be a Galois extension of number fields with Galois group $G$ (not necessarily satisfying AHC).  Let $\kp$ (resp.~$\ka$) be a nonzero prime ideal (resp.~nonzero ideal) of $\OK$, and let $\N_{K/\Q}\ka=|\OK/\ka|$.  Recall that $\Irr(G)$ is the set of irreducible complex representations of $G$.  We do not distinguish between a representation and its character.  We introduce
\[
\Rep(G)=\Big\{\bigoplus_{j=1}^m \chi_j\colon m\in\mathbb{N},~\chi_1,\ldots,\chi_m\in\Irr(G)\Big\}.
\]
Let $\mathbbm{1}\in\Irr(G)$ be the trivial character.

\subsection{Preliminaries}
\label{subsec:prelim}
We recall the basic definitions of Artin $L$-functions and some preliminary results following \cite{MurtyMurty}.  Let  $\kp$ be a prime ideal in $\cO_K$ and $\kP$ be a prime ideal in $\mathcal{O}_L$ lying above $\kp$. Let $D_\kP$ be the decomposition group of $\kP$, i.e., the set of elements in $G$ that fix $\kP$. Then there is a surjective group homomorphism
$D_\kP\to \operatorname{Gal}((\mathcal{O}_L/\kP) / (\mathcal{O}_K/\kp)),$
where $\OL/\kP$ and $\OK/\kp$ are the residue fields at $\kP$ and $\kp$, respectively. The kernel of this map is, by definition, the inertia group $I_\kP$. Hence, there is an isomorphism
\[
D_\kP/I_\kP\cong \operatorname{Gal}((\OL/\kP)/ (\OK/\kp)).
\]
The later group is the Galois group of an extension of finite fields and hence is generated by the Frobenius automorphism $x\mapsto x^{\Norm\kp}$. Thus, there exists a unique element $\varphi_\kP\in D_\kP/I_\kP$, the Frobenius element of $\kP$, that is sent to the Frobenius automorphism under the isomorphism.

The prime ideal $\kp$ is unramified if and only if the inertia group $I_\kP$ is trivial, in which case $\varphi_\kP$ is an element of $D_\kP$ (hence in $G$). Furthermore, different choices of $\kP$ lying above $\kp$ yield conjugate Frobenius elements in $G$. The corresponding conjugacy class of $G$ is what we call the Artin symbol of $\kp$, written as
\[
\varphi_\kp=\Big[\frac{L/K}{\kp}\Big].
\]
We emphasize that this is a well defined conjugacy class of $G$ only when $\kp$ is unramified.

Let $V$ be a finite dimensional complex vector space, and consider a representation of $G$ on $V$ with character $\chi\in\Rep(G)$.  If $I$ is a normal subgroup of $G$, we write $V^I$ for the subspace of all vectors in $V$ on which $I$ acts trivially. Then $G/I$ (and hence $G$) acts on $V^I$ as well. We denote the corresponding character of $G$ by $\tilde{\chi}\in\Rep(G)$.

Every representation of $G$ on $V$ restricts to a representation of $D_\kP$ on $V$. Since $I_\kP$ is a normal subgroup of $D_\kP$, we thus obtain a representation of $D_\kP/ I_\kP$ on $V^{I_\kP}$. If $\varphi_\kP\in D_\kP/I_\kP$ is the Frobenius element of $\kP$, then we have the characteristic polynomial
$$
\det(I-\varphi_\kP t;V^{I_\kP}),
$$
where $I$ the identity operator of $V^{I_{\kP}}$, and where $\varphi_\kP$ is understood as an automorphism of the vector space $V^{I_\kP}$.  The Artin $L$-function associated to $\chi$ is defined by
\begin{equation}
\label{eqn:definition of Artin $L$-functions}
    L(s,\chi,L/K)=\prod_\kp \frac{1}{\det(I-\varphi_\kP \N\kp^{-s};V^{I_\kP})},
\end{equation}
where the product is taken over all prime ideals $\kp$ in $\OK$. Therefore, the reciprocal of the Euler factor at $\kp$ is a polynomial of degree $\dim V^{I_\kP}$. Different choices of $\kP$ lying above $\kp$ yield the same characteristic polynomial, so \eqref{eqn:definition of Artin $L$-functions} is well defined.

We list some elementary properties of Artin $L$-functions. The proof of the following two propositions can be found in \cite[Proposition 10.4]{Neukirch}
\begin{proposition}
	\label{prop:elementary properties}
    Let $L/K$ be a Galois extension of number fields with Galois group $G$. Then the following are true:
    \begin{enumerate}
        \item $L(s,\mathbbm{1}_G,L/K)=\zeta_K(s)$.
        \item If $\phi,\psi\in\Rep(G)$, then $L(s,\phi\oplus \psi,L/K)=L(s,\phi,L/K)L(s,\psi,L/K)$.
        \item Let $N$ be a normal subgroup of $G$ and $\mathop{\mathrm{Proj}}\colon G\to G/N$ be the quotient map. If $\phi\in\Rep(G/N)$, then $\phi\circ \mathop{\mathrm{Proj}}\in \Rep(G)$, and $L(s,\phi,L^N/K)=L(s,\phi\circ \mathop{\mathrm{Proj}},L/K)$.
    \end{enumerate}
\end{proposition}

Another fundamental property of Artin $L$-functions is their invariance under induction.
\begin{proposition}
\label{prop:inductive invariance}
Let $L/K$ be a Galois extension of number fields with Galois group $G$.  If $H$ be a subgroup of $G$ and $\chi$ is a representation of $H$, then $L(s,\Ind_H^G\chi,L/K)=L(s,\chi,L/L^H)$.
\end{proposition}

Applying \cref{prop:inductive invariance,prop:elementary properties} to the regular representation
\[
\reg_G=\Ind_{\{e_G\}}^G\mathbbm{1}_{\{e_G\}}=\bigoplus_{\chi\in\Irr(G)}\chi(1)\chi
\]
of $G$, we obtain the important factorization:
\begin{equation}\label{eqn:factorization of zeta_L}
	\zeta_L(s)=\prod_{\chi\in\Irr(G)}L(s,\chi,L/K)^{\chi(1)}.
\end{equation}

\subsection{The Artin conductor}

For each $\chi\in\Rep(G)$, there is a distinguished ideal $\f_\chi\subseteq\OK$, the Artin conductor of $\chi$, defined as follows.  Let $V$ be the vector space associated to $\chi$. Let $\kp$ be a prime ideal in $\OK$ and let $\kP$ be a prime ideal in $\OL$ that lies above $\kp$. Let $G_0\supset G_1\supseteq G_2\supseteq \cdots$ be the filtration of ramification groups corresponding to $\kP$, where $G_0$ is the inertia group $I_\kP$.  The Artin conductor of $\chi$ is
\begin{equation}\label{eqn:def of Artin conductor}
\f_\chi=\prod_{\kp}\kp^{\mathrm{ord}_{\kp}(\kf_{\chi})},\qquad \mathrm{ord}_{\kp}(\kf_{\chi})=\sum_{i=0}^\infty\frac{|G_i|}{|G_0|}(\chi(1)-\dim V^{G_i}).
\end{equation}
(The number $\mathrm{ord}_{\kp}(\kf_{\chi})$ is always a nonnegative integer.) The conductor-discriminant formula states that if $\mathfrak{D}_{L/K}$ is the relative discriminant of $L/K$, then
\begin{equation}
\label{eqn:conductor discriminant formula}
\mathfrak{D}_{L/K}=\prod_{\chi\in\Irr(G)}\kf_{\chi}^{\chi(1)},\qquad D_L/D_K^{|G|}= \prod_{\chi\in\Irr(G)}\N_{K/\Q}\kf_{\chi}^{\chi(1)}.
\end{equation}

\begin{lemma}
\label{lem:p_divides_conductor}
Let $V$ be a finite dimensional complex vector space and $\chi$ be the character of a representation of $G$ on $V$. Let $\kp$ be a prime ideal in $\OK$ and $\kP$ be a prime ideal in $\OL$ lying over $\kp$. Then $\dim V^{I_\kP}<\dim V$ if and only if $\kp\mid \f_\chi$.
\end{lemma}
\begin{proof}
Recall \eqref{eqn:def of Artin conductor}.  If $\dim V^{I_\kP}<\dim V=\chi(1)$, then $\mathrm{ord}_{\kp}(\kf_{\chi})\geq 1$, whence $\kp\mid\kf_\chi$. Conversely, if $\dim V^{I_\kP}=\dim V$, i.e., $V^{I_\kP}=V$, then $I_\kP=G_0$ acts on $V$ trivially.  Thus, each $G_i$ must also act on $V$ trivially, i.e., $V^{G_i}=V$ for all $i\ge0$. This implies that $\mathrm{ord}_{\kp}(\kf_{\chi})=0$.
\end{proof}

Here is a useful result on Artin conductors associated to tensor products.

\begin{lemma}
\label{lem:tensor conductor bound}
If $\chi_1,\chi_2\in\Irr(G)$, then $\f_{\chi_1 \otimes \chi_2}\mid \f_{\chi_1}^{\chi_2(1)} \f_{\chi_2}^{\chi_1(1)}$.
\end{lemma}
\begin{proof}
A stronger result is proved in \cite[Lemma 6.6]{LOTZ}.
\end{proof}

\subsection{Frobenius eigenvalues}

\cref{lem:p_divides_conductor} states that $\dim V^{I_\kP}<\dim V$ if and only if $\kp$ divides the Artin conductor $\f_\chi$.  With this in mind, if $\{\alpha_{1,\chi}(\kp),\ldots,\alpha_{\dim V^{I_{\kP}},\chi}(\kp)\}$ is an unordered list of the eigenvalues of the action of $\varphi_\kP$ on $V^{I_{\kP}}$, then we define
\begin{equation}
\label{eqn:set of local roots}
A_{\chi}(\kp) = \begin{cases}
\{\alpha_{1,\chi}(\kp),\ldots,\alpha_{\dim V,\chi}(\kp)\}&\mbox{if $\kp\nmid\kf_{\chi}$,}\\
\{\alpha_{1,\chi}(\kp),\ldots,\alpha_{\dim V^{I_{\kP}},\chi}(\kp),\underbrace{0,0,\ldots,0}_{\textup{$\dim V-\dim V^{I_{\kP}}$ times}}\}&\mbox{if $\kp\mid\kf_{\chi}$.}
\end{cases}
\end{equation}
Therefore, we may always write (with no ordering)
\[
A_{\chi}(\kp)=\{\alpha_{1,\chi}(\kp),\ldots,\alpha_{\chi(1),\chi}(\kp)\}.
\]
We shall refer to the numbers $\alpha_{j,\chi}(\kp)$ either as the local roots or the Frobenius eigenvalues of $\chi$ at $\kp$. Since $G$ is finite, we have that
\begin{equation}
\label{eqn:GRC}
|\alpha_{j,\chi}(\kp)| \begin{cases}
    =1&\mbox{if $\kp\nmid\kf_{\chi}$,}\\
    \in\{0,1\}&\mbox{if $\kp\mid\kf_{\chi}$.}
\end{cases}
\end{equation}
Therefore, we can rewrite \eqref{eqn:definition of Artin $L$-functions} as
\begin{equation}
\label{eqn:eigenvalue definition of Artin $L$-functions}
\begin{aligned}
    L(s,\chi,L/K)&=\prod_\kp L_{\kp}(s,\chi,L/K)=\sum_{\ka}\frac{\lambda_{\chi}(\ka)}{\N\ka^s},\\
    L_{\kp}(s,\chi,L/K)&=\prod_{j=1}^{\chi(1)}\frac{1}{1 - \alpha_{j,\chi}(\kp) \N\kp^{-s}} = 1+\sum_{j=1}^{\infty}\frac{\lambda_{\chi}(\kp^j)}{\N\kp^{js}}.\hspace{-1mm}
\end{aligned}
\end{equation}
It follows from \eqref{eqn:GRC} that $L(s,\chi)$, in both its Euler product and Dirichlet series expansions, converges absolutely for $\re(s)>1$.

Define
\[
a_{\chi}(\ka)=\begin{cases}
\sum_{j=1}^{\chi(1)}\alpha_{j,\chi}(\kp)^{\ell} &\mbox{if there exists a prime ideal $\kp$ and $\ell\in\mathbb{N}$ such that $\ka=\kp^{\ell}$,}\\
0&\mbox{otherwise}
\end{cases}
\]
and
\[
\Lambda_K(\ka)=\begin{cases}
\log \N\kp&\mbox{if there exists a prime ideal $\kp$ and $\ell\in\mathbb{N}$ such that $\ka=\kp^{\ell}$,}\\
0&\mbox{otherwise.}
\end{cases}
\]
These numbers arise in the Dirichlet series expansion
\begin{equation}\label{eqn:a_phi definition}
\frac{L'}{L}(s,\chi,L/K) = \sum_{\ka}\frac{a_{\chi}(\ka)\Lambda_K(\ka)}{\N\ka^{s}},\qquad \re(s)>1.
\end{equation}
Per \eqref{eqn:GRC}, we have the uniform bound
\begin{equation}\label{eqn:trivial bound on a_chi}
	|a_\chi(\ka)| \leq \chi(1).
\end{equation}
Our next result helps us to efficiently handle $a_{\chi_1\otimes\chi_2}(\ka)$, where $\chi_1,\chi_2\in\Rep(G)$.

\begin{lemma}\label{lem:trace-equation}
    Let $\mathcal{G}$ be a finite group and $\chi\in\Rep(\mathcal{G})$ be the character of a representation on a finite dimensional complex vector space $U$. Let $I$ be a normal subgroup of $\mathcal{G}$, so that $\mathcal{G}$ acts on $U^I$ as well, and let $\tilde{\chi}$ be the character of the corresponding representation of $\mathcal{G}$ (equivalently, $\mathcal{G}/I)$ on $U^I$. Then for any $g\in \mathcal{G}$, one has
	\begin{equation}\label{eqn:trace-equation}
		\tilde{\chi}(g)=\frac{1}{|I|}\sum_{\alpha\in I}\chi(g\alpha).
	\end{equation}
\end{lemma}

\begin{proof}
	Choose a basis $\{\mathbf{u}_1,\ldots,\mathbf{u}_r\}$ of $U^I$ and extend it to a basis $\{\mathbf{u}_1,\ldots,\mathbf{u}_r,\mathbf{u}_{r+1},\ldots,\mathbf{u}_n\}$ of $U$. Let $h=g\sum_{\alpha\in I}\alpha$.  The right-hand side of \eqref{eqn:trace-equation} is exactly $1/|I|$ times the trace of $h$ on $U$. For each $u_i$ ($1\le i\le r$), we have $h\mathbf{u}_i=g\sum_{\alpha\in I}\alpha \mathbf{u}_i=|I|g\mathbf{u}_i$, because each $\alpha\in I$ fixes $\mathbf{u}_i$ by assumption.  Therefore, the action of $h$ on $U^I$ is just $|I|$ times the action of $g$ on ${U^I}$.  On the other hand, for each $\mathbf{u}_j$ ($r+1\le j\le n$), the vector
	$\sum_{\alpha\in I}\alpha \mathbf{u}_j$
	lies in $U^I$.  Indeed, if $\beta\in I$ then
	$$
	\beta\sum_{\alpha\in I}\alpha \mathbf{u}_j=\sum_{\alpha\in I}\beta\alpha \mathbf{u}_j=\sum_{\alpha\in I}\alpha \mathbf{u}_j,\qquad\textup{hence}\quad h\mathbf{u}_j=g\sum_{\alpha\in I}\alpha \mathbf{u}_j\in U^I.
	$$
	This means that the matrix representation of $h$ is of the form
	$$
	\begin{pmatrix}
|I|\cdot g|_{U^I} & * \\
0 & 0
\end{pmatrix}
	$$
	with trace $|I|\cdot\tilde{\chi}(g)$, where $g|_{U^I}$ denotes the restriction of $g$ to $U^I$. This proves the lemma.
\end{proof}

\begin{lemma}\label{lem:nonnegativity of dirichlet coefficients}
   If $\chi\in\Rep(G)$, then $a_{\chi\otimes\bar\chi}(\ka)\ge0$.
\end{lemma}
\begin{proof}
Let $\kp$ be a prime ideal of $\cO_K$, and let $\kP$ be  a prime ideal of $\cO_L$ that lies above $\kp$.  Let $V$ be the underlying finite-dimensional complex vector space on which $G$ acts via the representation $\chi$. Then $G$ also acts on $V\otimes V^*$ via $\chi\otimes\bar\chi$. Therefore, the decomposition group $D_\kP$ acts on $V\otimes V^*$. Since $I_\kP$ is a normal subgroup of $D_\kP$, we know that $D_\kP/I_\kP$ acts on $W=(V\otimes V^*)^{I_\kP}$. If $\varphi_\kP\in D_\kP/I_\kP$ is the Frobenius element of $\kP$ , then $a_{\chi\otimes\bar\chi}(\kp^{\ell})$ equals the sum of $m$th powers of eigenvalues of $\varphi_\kP|_W$, which equals the trace of $\varphi_\kP^{\ell}|_W$.
    Let $\varphi\in D_\kP$ be any element whose image under the quotient map $D_\kP\to D_\kP/I_\kP$ is $\varphi_\kP$. Now, applying \cref{lem:trace-equation} with $\mathcal{G}=D_\kP$, $I=I_\kP$, $U=V\otimes V^*$, and $g=\varphi^{\ell}$, we obtain
    \[
    a_{\chi\otimes\bar\chi}(\kp^{\ell})=\frac{1}{|I_\kP|}\sum_{\alpha\in I_{\kP}}(\chi\otimes\bar\chi)(\varphi^{\ell}\alpha)=\frac{1}{|I_\kP|}\sum_{\alpha\in I_{\kP}}\chi(\varphi^{\ell}\alpha)\overline{\chi(\varphi^{\ell}\alpha)}=\frac{1}{|I_\kP|}\sum_{\alpha\in I_{\kP}}|\chi(\varphi^{\ell}\alpha)|^2\ge0,
    \]
as desired.  If $\ka$ is not of the form $\kp^{\ell}$, then $a_{\chi\otimes\bar{\chi}}(\ka)=0$.
\end{proof}

\begin{lemma}
\label{lem:a_inequality}
If $\chi_1,\chi_2\in\Rep(G)$, then $|a_{\chi_1\otimes\chi_2}(\ka)|^2\leq a_{\chi_1\otimes\bar{\chi}_1}(\ka)a_{\chi_2\otimes\bar{\chi}_2}(\ka)$.
\end{lemma}
\begin{proof}
	Again, we may assume that $\ka=\kp^{\ell}$, where $\kp$ is a prime ideal in $\OK$. Let $V_1$ (resp. $V_2$) be the complex vector space associated to $\chi_1$ (resp. $\chi_2$). Following the same steps in \cref{lem:nonnegativity of dirichlet coefficients} and applying \cref{lem:trace-equation} with $\mathcal{G}=D_\kP$, $I=I_\kP$, $U=V_1\otimes V_2$, and $g=\varphi^\ell$, we obtain
    \begin{equation}
    \label{eqn:ap_formula}
    a_{\chi_1\otimes\chi_2}(\kp^{\ell})=\frac{1}{|I_\kP|}\sum_{\alpha\in I_{\kP}}(\chi_1\otimes\chi_2)(\varphi^\ell\alpha)=\frac{1}{|I_\kP|}\sum_{\alpha\in I_{\kP}}\chi_1(\varphi^\ell\alpha)
\chi_2(\varphi^\ell\alpha).
    \end{equation}
Now, applying the Cauchy--Schwarz inequality, we obtain the bound
\begin{align*}
|a_{\chi_1\otimes\chi_2}(\kp^{\ell})|^2
&\leq  \frac{1}{|I_{\kP}|^2}\Big(\sum_{\alpha\in I_{\kP}}|\chi_1(\varphi^\ell\alpha)|^2\Big)\Big(\sum_{\alpha\in I_{\kP}}|\chi_2(\varphi^\ell\alpha)|^2\Big)\\
&=  \Big(\frac{1}{|I_{\kP}|}\sum_{\alpha\in I_{\kP}}|\chi_1(\varphi^\ell\alpha)|^2\Big)\Big(\frac{1}{|I_{\kP}|}\sum_{\alpha\in I_{\kP}}|\chi_2(\varphi^\ell\alpha)|^2\Big)=a_{\chi_1\otimes\bar{\chi}_1}(\kp^{\ell})a_{\chi_2\otimes\bar{\chi}_2}(\kp^{\ell}).\hspace{4mm}\qedhere
\end{align*}
\end{proof}

\subsection{Character orthogonality}
\label{subsec:orthogonality}

For $\chi_1,\chi_2\in\Rep(G)$, we define
    \begin{equation}
    \label{eqn:inner_product}
        \langle \chi_1, \chi_2\rangle=\frac{1}{|G|}\sum_{g\in G}\chi_1(g)\overline{\chi}_2(g).
    \end{equation}
Character orthogonality for $G$ is the statement that if $\chi_1,\chi_2\in\Irr(G)$, then
\[
\langle \chi_1,\chi_2\rangle = \begin{cases}
    1&\mbox{if $\chi_1=\chi_2$,}\\
    0&\mbox{if $\chi_1\neq\chi_2$.}
\end{cases}
\]
Maschke proved that if $\chi\in\Rep(G)$, then there exist $m\in\mathbb{N}$ and $\chi_1,\ldots,\chi_m\in\Irr(G)$ (the irreducible constituents of $\chi$) such that $\chi = \bigoplus_{j=1}^m \chi_j$.  In particular, if $\chi_1,\chi_2\in\Irr(G)$, then such a decomposition exists for $\chi_1\otimes\chi_2$.

Let $\chi,\chi'\in\Rep(G)$ with direct sum decompositions $\chi = \bigoplus_{j=1}^m \chi_j$ and $\chi' = \bigoplus_{k=1}^{n} \psi_{k}$ of irreducible constituents.  The (conjugate) bilinearity of the inner product implies that
\[
\langle \chi,\chi'\rangle = \sum_{j=1}^m \sum_{k=1}^n \langle \chi_j,\psi_k\rangle.
\]
If $\psi\in\Rep(G)$ and $\tau\in\Irr(G)$, then the multiplicity of $\tau$ as a constituent of $\psi$ is $\langle \psi,\tau\rangle$.

\subsubsection{Irreducible constituents}

\begin{lemma}\label{lem:tensor-contains-triv}
    If $\chi,\psi\in\Irr(G)$, then
    \begin{enumerate}
        \item $\mathbbm{1}$ is a constituent of  $\chi\otimes\psi$ if and only if $\psi=\bar\chi$,
        \item the multiplicity of $\mathbbm{1}$ as a constituent of $\chi\otimes\overline{\chi}$ is one, and
        \item if $\chi(1)\geq 2$, then $\chi\otimes\overline{\chi}$ contains an irreducible constituent $\tau\notin \{\chi,\overline{\chi},\mathbbm{1}\}$. In this case, $\chi$ is an irreducible constituent of $\chi\otimes\tau$, but $\mathbbm{1}$ is not.
    \end{enumerate}
\end{lemma}

\begin{proof}
The first and second claims follow from the computation
\[
\langle \chi\otimes\psi,\mathbbm{1}\rangle = \langle\chi,\overline{\psi}\rangle = \begin{cases}
    1&\mbox{if $\chi=\overline{\psi}$,}\\
    0&\mbox{otherwise.}
\end{cases}
\]
For the third claim, assume that $\chi\in\Irr(G)$ and $\chi(1)\geq 2$.  Let $\tau\in\Irr(G)$ be an irreducible constituent of $\chi\otimes\bar{\chi}$.  Suppose to the contrary that there exist integers $J_1,J_2\geq 0$ such that $J_1+J_2\geq 1$ and
\[
\chi\otimes\overline{\chi}=\mathbbm{1}\oplus\Big(\bigoplus_{j=1}^{J_1}\chi\Big)\oplus\Big(\bigoplus_{j=1}^{J_2}\overline{\chi}\Big).
\]
Comparing the dimensions of both sides, we arrive at $\chi(1)^2 = 1+(J_1+J_2)\chi(1)$.  This implies that $\gcd(\chi(1)^2,\chi(1))=1$, in which case $\chi(1)=1$, a contradiction.  Therefore, $\tau$ must lie in $\Irr(G)-\{\chi,\bar\chi,\mathbbm{1}\}$.  Since $\tau\ne\bar\chi$, we have $\langle \chi\otimes\tau,\mathbbm{1}\rangle=\langle \tau,\bar\chi\rangle=0$.  Since $\tau$ is a constituent of $\chi\otimes\bar\chi$, we have $\langle \chi\otimes\tau,\chi\rangle=\langle \tau,\chi\otimes\bar\chi\rangle\geq 1$.  This finishes the proof of the third claim.
\end{proof}

\subsubsection{The Artin symbol}

For $\kp$ a prime ideal in $\cO_K$, let $\kP$ be any prime ideal in $\cO_L$ lying above $\kp$, and let $I_\kP$ be its inertia group.  For $\ell\in\mathbb{N}$, define
\begin{equation}
\label{eqn:tau_def}
\tau_C(\kp^{\ell})=
\frac{1}{|I_{\mathfrak{P}}|} \sum_{\alpha \in I_{\mathfrak{P}}} \frac{|C|}{|G|} \sum_{\chi \in \operatorname{Irr}(G)} \bar{\chi}(C) \chi(\varphi_\kP^{\ell} \alpha).
\end{equation}
By character orthogonality, we always have $0\le \tau(\kp^{\ell})\le 1$. Furthermore, if $\kp$ is unramified in $L$, then $I_\kP$ is trivial, and in this case we have
\begin{equation}
    \tau(\kp^{\ell})=\frac{|C|}{|G|} \sum_{\chi \in \operatorname{Irr}(G)} \bar{\chi}(C) \chi(\varphi_\kP^{\ell})= \begin{cases}1 & \mbox{if $[\frac{L / K}{\mathfrak{p}}]^{\ell}=C$,} \\ 0 & \mbox{otherwise.}\end{cases}
\end{equation}
Thus, when $\kp$ is unramified, $\tau(\kp)$ equals the indicator function of $\kp$ having Artin symbol $C$.  For all other prime powers $\kp^{\ell}$, we have that $\tau(\kp^{\ell})\in[0,1]$.

\subsection{Meromorphic continuation and functional equation}
For $\chi\in\Rep(G)$, write
\[
\cA_{\chi}=D_K^{\chi(1)}\N\kf_{\chi},\qquad \delta(\chi)=\langle \chi,\mathbbm{1}_G\rangle
\]
(see \eqref{eqn:inner_product}).  Let  $a=a(\chi)$ be the dimension of the +1 eigenspace of complex conjugation.   For each archimedean place $v$ of $K$ whose corresponding completion is $K_v$, we define
\[
L_v(s, \chi,L/K)= \begin{cases}\Gamma_{\mathbb{R}}(s)^{\chi(1)} \Gamma_{\mathbb{R}}(s+1)^{\chi(1)} & \text { if } K_v=\mathbb{C},
\\
\Gamma_{\mathbb{R}}(s)^a \Gamma_{\mathbb{R}}(s+1)^{\chi(1)-a} & \text { if } K_v=\mathbb{R},\end{cases}\qquad \Gamma_{\R}(s) = \pi^{-\frac{s}{2}}\Gamma\Big(\frac{s}{2}\Big).
\]
We also define the numbers $\mu_\chi(j)\in\{0,1\}$ by
\[
L_\infty(s,\chi,L/K)=\prod_{v\text{ archimedian}}L_v(s,\chi,L/K)=\prod_{j=1}^{\chi(1)n_K}\Gamma_\R(s+\mu_\chi(j)).
\]

The completed $L$-function
\[
\Lambda(s,\chi,L/K)=(s(1-s))^{\delta(\chi)}\cA_{\chi}^{s/2}L_\infty(s,\chi,L/K)L(s,\chi,L/K)
\]
admits a meromorphic continuation to $\mathbb{C}$, and there exists a number $W(\chi) \in \mathbb{C}$ of modulus $1$ such that if $\Lambda(s,\chi)$ is holomorphic at $s$, then
\begin{equation}\label{eqn:functional equation}
	\Lambda(s, \chi,L/K)=W(\chi) \Lambda(1-s, \bar{\chi},L/K).
\end{equation}
Here, $\bar{\chi}$ is the complex conjugate of $\chi$.

\subsection{The character $\chi_1$}

Recall the definitions of $\beta_1$ and $\chi_1$ in \eqref{eqn:def of beta_1} and \eqref{eqn:def of chi_1}. Note that since $\zeta_L(s)$ always has real trivial zeros, $\beta_1$ is well defined. If $\beta_1$ is not a simple zero, then $\chi_1:=\mathbbm{1}_G$. To show that $\chi_1$ is well defined when $\beta_1$ is a simple zero, we need the following theorem from Stark.
\begin{theorem}[{\cite[Theorem 3]{Stark}}]
\label{thm:Stark_zero}
    Let $L/K$ be a Galois extension of number fields with Galois group $G$, and let $\chi\in\Irr(G)$. If $\rho$ is a simple zero of $\zeta_L(s)$, then $L(s, \chi, L/K)$ is holomorphic at $s=\rho$.  Also, there is a field $F$, cyclic over $K$ and contained in $L$, such that for any tower of fields $K\subseteq E\subseteq L$, $\zeta_E(\rho)=0$ if and only if $F \subseteq E$. If $\rho$ is real, then $F$ is either $K$ itself or quadratic over $K$.
\end{theorem}

The well-definedness of $\chi_1$ (when $\beta_1$ is simple) now follows, as the next lemma shows:
\begin{lemma}
    \label{eqn:chi_1 well defined}
	Let $\beta_1$ be as in \eqref{eqn:def of beta_1}. If $\beta_1$ is simple, then there exists a unique character $\chi_0\in \Irr(G)$ such that $\beta_1$ is a zero of $L(s,\chi_0,L/K)$. Furthermore, $\chi_0=\bar{\chi}_0$ and $\chi_0(1)=1$.
\end{lemma}
\begin{proof}
	Apply \cref{thm:Stark_zero} with $\rho=\beta_1$. Since $\beta_1$ is a simple zero and each Artin $L$-function $L(s,\chi,L/K)$ is analytic at $s=\beta_1$, we see from the factorization of $\zeta_L(s)$ in \eqref{eqn:factorization of zeta_L}
	that there exists a unique $\chi_0\in\Irr(G)$ such that $L(\beta_1,\chi_0,L/K)=0$. It remains to show that $\chi_0$ is real and one-dimensional.  Let $F$ be the unique intermediate subfield of $L/K$ with the desired property as defined in \cref{thm:Stark_zero}. Since $\beta_1$ is real, $F$ is either $K$ or quadratic over $K$. If $F=K$, then $\zeta_K(\beta_1)=0$, and we see that $\chi_0=\mathbbm{1}_G$. If $F\ne K$, then $\Gal(L/F)$ is a normal subgroup of index 2 in $G$ since $F/K$ is quadratic. Thus, $\Gal(F/K)$ is cyclic of order 2, so it has a unique nontrivial character $\phi_1$, which is necessarily real and one-dimensional. Let
\[
\mathrm{Proj}\colon G\to G/\Gal(L/F)\cong \Gal(F/K)
\]
be the quotient map. Then $\phi_1':=\phi_1\circ \mathrm{Proj}$ is a character of $G$. Since $\zeta_F(\beta_1)=0$ and $\zeta_K(\beta_1)\ne0$, we see from the factorization $\zeta_F(s)=\zeta_K(s)L(s,\phi_1,F/K)$ that $L(\beta_1,\phi_1',L/K)=L(\beta_1,\phi_1,F/K)=0$ by \cref{prop:elementary properties}. Hence, $\chi_0=\phi_1'$. In either case, $\chi_0$ is real and one-dimensional. This finishes the proof.
\end{proof}

Our definition of $\beta_1$ and $\chi_1$ has several advantages. First, it does not rely on any zero free region of Artin L-functions. Second, it is well-defined for any Galois extension of number fields, without assuming AHC. Third, the notion of $\chi_1$ is well-behaved under base change (see \cref{lem:restriction of chi_1} below).

\section{Consequences of analytic continuation}
\label{sec:AHC}

Let $L/K$ be a Galois extension of number fields with Galois group $G$.  In this section, we shall assume that $L/K$ satisfies AHC, so that if $\chi\in\Rep(G)$, then $\Lambda(s,\chi)$ is entire of order 1. It follows that there exist $B_1(\chi),B(\chi)\in\mathbb{C}$ such that $\Lambda(s,\chi,L/K)$ has the Hadamard factorization
\begin{equation}\label{eqn:hadamard}
    \Lambda(s,\chi,L/K) = e^{B_1(\chi) + B(\chi)s} \prod_{\Lambda(\rho,\chi,L/K)=0} \Big(1 - \frac{s}{\rho}\Big) e^{s/\rho}.
\end{equation}
The zeros $\rho=\beta+i\gamma$ of $\Lambda(s,\chi)$ are the nontrivial zeros of $L(s,\chi)$.  In all sums, products, and counts over nontrivial zeros $\rho$ of $L(s,\chi)$, the zeros are counted with multiplicity.

Taking the logarithmic derivative of \eqref{eqn:hadamard}, we obtain
\begin{equation*}
    B(\chi)+\sum_\rho\Big(\frac{1}{s-\rho}+\frac{1}{\rho}\Big)=\frac{\log\cA_{\chi}}{2}+\frac{L'}{L}(s, \chi,L/K)+\frac{L_{\infty}'}{L_{\infty}}(s,\chi,L/K)+\frac{\delta(\chi)}{s}+\frac{\delta(\chi)}{s-1}.
\end{equation*}
From \cite[Proposition 5.7(3)]{IK}, we have
\begin{equation*}
    \Re B(\chi) = -\sum_{\rho} \Re\frac{1}{\rho},
\end{equation*}
hence
\begin{equation}
\label{eqn:L'/L-equality}
\sum_{\rho} \Re\Big(\frac{1}{s-\rho}\Big) = \frac{\log\cA_{\chi}}{2}+\Re\Big(\frac{L'}{L}(s,\chi,L/K) + \frac{L_{\infty}'}{L_{\infty}}(s,\chi,L/K)+\frac{\delta(\chi)}{s} + \frac{\delta(\chi)}{s-1}\Big).
\end{equation}

\subsection{Counting zeros}

In what follows, we use the notation in \eqref{eqn:degree_q_def}.

\begin{lemma}[Lemma, p. 1203 of \cite{HIJTS}]
\label{lem:Stirling}
If $\Re(z)\geq \frac{1}{2}$, then
\[
\Re\Big(\frac{\Gamma_{\R}'}{\Gamma_{\R}}(z)\Big)\leq \frac{\log|z|}{2}-\frac{\gamma_{\Q}+\log \pi}{2},
\]
where $\gamma_\Q=0.57721\ldots$ is the Euler-Mascheroni constant.
\end{lemma}

\begin{lemma}
\label{lem:Riemann-vonMangoldt}
    If $t \in \mathbb{R}$ and $\chi \in \Irr(G)$, 
    then    
\[
\#\{\beta+i \gamma: \beta \geq 0,|\gamma-t| \leq 1, L(\beta+i \gamma, \chi)=0\} \ll \log(D_K^{\chi(1)} \mathrm{N}\kf_{\chi}(|t|+3)^{\chi(1) n_K}).
\]
\end{lemma}
\begin{proof}
The proof is similar to that of \cite[Proposition 5.7(1)]{IK}. 
\end{proof}

\begin{lemma}\label{lem:bound of Dirichlet series}
If $\chi,\phi\in\Irr(G)$ and $0<\eta<1$, then
\[
\sum_{\ka}\frac{|a_{\chi\otimes\phi}(\ka)|\Lambda_K(\ka)}{\N\ka^{1+\eta}}\leq \frac{1}{\eta}+\frac{\log q_{L/K}}{2}.
\]
\end{lemma}
\begin{proof}
First, we address the case when $\phi=\bar{\chi}$.  If $\chi(1)=1$ and $K=\Q$, then $\chi\otimes\bar\chi=\mathbbm{1}_G$ and this reduces to the well-known inequality
\[
-\frac{\zeta'}{\zeta}(1+\eta)\leq\frac{1}{\eta},\qquad 0<\eta\leq 1.
\]
Otherwise, we have that $|a_{\chi\otimes\phi}(\kp^{\ell})|=a_{\chi\otimes\bar{\chi}}(\kp^{\ell})\ge 0$, and the sum to be estimated equals
\begin{multline*}
-\frac{L'}{L}(1+\eta,\chi\otimes\bar{\chi},L/K) = \frac{1}{\eta}+\frac{1}{1+\eta}+\frac{L_{\infty}'}{L_{\infty}}(1+\eta,\chi\otimes\bar{\chi},L/K)+\frac{\log (D_K^{\chi(1)^2}\kf_{\chi\otimes\bar{\chi}})}{2}\\
-\sum_{\Lambda(\rho,\chi\otimes\bar{\chi},L/K)=0}\re\Big(\frac{1}{1+\eta-\rho}\Big).
\end{multline*}
Note that since either $\chi(1)>1$ or $K\neq\Q$, $L_{\infty}(s,\chi\otimes\bar{\chi},L/K)$ has at least two Gamma factors.  Therefore, since $\mu_{\chi\otimes\bar{\chi}}(j)\in\{0,1\}$ and $0<\eta\leq 1$, we find using \cref{lem:Stirling} that
\[
\frac{1}{1+\eta}+\frac{L_{\infty}'}{L_{\infty}}(1+\eta,\chi\otimes\bar{\chi},L/K)\leq 0.
\]
Finally, since $\re(\rho)\leq 1$, we find that $-\re(\frac{1}{1+\eta-\rho})\leq 0$.  We now conclude from \cref{lem:tensor conductor bound} that
\[
-\frac{L'}{L}(1+\eta,\chi\otimes\bar{\chi},L/K)\leq \frac{1}{\eta}+\frac{\log q_{L/K}}{2}.
\]
The case where $\phi\neq\bar{\chi}$ now follows from \cref{lem:a_inequality}, the Cauchy--Schwarz inequality, and our work above for $\chi\otimes\bar{\chi}$.
\end{proof}

\begin{lemma}
\label{lem:Linnik}
Let $s=\sigma+it$ with $\sigma\geq 1$.  If $0<\eta\leq 1$ and $\chi,\phi\in\Irr(G)$, then
\begin{align*}
N(\eta;s,\chi\otimes\phi)&=\#\{\rho=\beta+i\gamma\colon \Lambda(s,\chi\otimes\phi,L/K)=0,~|s-\rho|\leq\eta\}\\
&\leq 8+4\eta \log(q_{L/K} (|t|+3)^{d_G^2 n_K}).
\end{align*}
\end{lemma}
\begin{proof}
It suffices to consider $\sigma=1$.  We begin with the estimate
\[
\frac{\#\{\rho\colon \Lambda(\rho,\chi\otimes\phi,L/K)=0,~|\rho-(1+it)|\leq \eta\}}{4\eta}\leq \sum_{\Lambda(\rho,\chi\otimes\phi,L/K)=0}\frac{1+\eta-\beta}{(1+\eta-\beta)^2 + (t-\gamma)^2}.
\]
The summands are $\re(1/(1+\eta+it-\rho))$, which are nonnegative.  Thus, by a calculation similar to \cref{lem:bound of Dirichlet series}, the right-hand side is
\[
\leq \frac{\log q_{L/K}}{2}+\frac{1}{2}\re\Big(\frac{L_{\infty}'}{L_{\infty}}(1+\eta+it,\chi\otimes\phi,L/K)\Big)+\sum_{\ka}\frac{|a_{\chi\otimes\phi}(\ka)|\Lambda_K(\ka)}{\N\ka^{1+\eta}}+\frac{1}{\eta}+1.
\]
By \cref{lem:Stirling,lem:bound of Dirichlet series}, the preceding display is $\leq \frac{2}{\eta}+\log q_{L/K}+d_G^2 n_K\log(|t|+3)$.
\end{proof}

\subsection{Zero-free region}

We establish the best known zero-free region for $\zeta_L(s)$ when $L/K$ is a Galois extension of number fields such that $\Gal(L/K)$ satisfies AHC.
\begin{theorem}
\label{thm:ZFR}
Let $L/K$ be a Galois extension of number fields with Galois group $G$ satisfying AHC.  There exists a constant $\Cl[abcon]{ZFR}\in(0,\frac{1}{2})$ such that
\[
\zeta_L(\sigma+it) = \prod_{\chi\in\Irr(G)}L(\sigma+it,\chi,L/K)^{\chi(1)}
\]
has at most one zero in the region
\begin{equation}
\label{eqn:ZFR}
\sigma\geq 1 - \frac{\Cr{ZFR}}{\log(q_{L/K}(|t|+3)^{d_G^2 n_K})}.
\end{equation}
If a zero exists in this region, then it is necessarily $\beta_1 = \beta_{1,L}$ defined by \eqref{eqn:def of beta_1}, in which case $\beta_1$ is simple, $\beta_1$ satisfies the bound
\[
\beta_1\leq 1-\Cr{ZFR_beta0}(2\cQ_{L/K})^{-3},
\]
and there is a unique character $\chi_0=\chi_{0,L/K}$ (defined in \eqref{eqn:chi_0_def}) such that $L(\beta_1,\chi_0,L/K)=0$.
\end{theorem}
\begin{proof}
First, we consider the contribution from $\chi\in\Irr(G)$ with $\chi(1)=1$.  If $\chi\in\Irr(G)$ and $\chi(1)=1$, then $L(s,\chi,L/K)$ is a Hecke $L$-function whose completed $L$-function $\Lambda(s,\chi,L/K)$ is already known to be entire.  Therefore, the claimed zero-free region follows immediately (in a stronger form) from the work of Lagarias, Montgomery, and Odlyzko \cite[Lemma 2.3]{LMO} and Weiss \cite[Theorem 1.9 and its proof]{Weiss}.  The upper bound on $\beta_1$ stated in \cite{Weiss} follows from the work of Stark \cite[Theorem 1]{Stark} and the conductor-discriminant formula.
\end{proof}

  We cite the following zero-free region due to Lagarias, Montgomery, Odlyzko, Stark, and Weiss.
\begin{lemma}
\label{lem:ZFR_1-dim}
There exists a constant $\Cl[abcon]{ZFR_LMO}$ such that
\[
\prod_{\substack{\chi\in\Irr(G) \\ \chi(1)=1}}L(\sigma+it,\chi,L/K)
\]
has at most one exceptional zero $\beta_0$ in the region
\[
\sigma\geq 1-\frac{\Cr{ZFR_LMO}}{\log(q_{L/K}(|t|+3)^{n_K})}.
\]
If $\beta_0$ exists, then $\beta_0=\beta_1$, where $\beta_1$ is defined in \eqref{eqn:def of beta_1}, and $\beta_0$ is simple. Furthermore, there exists a 
constant $\Cl[abcon]{ZFR_beta0}\in(0,1)$ such that if $\beta_0$ exists, then
\[
\beta_0\leq 1-\Cr{ZFR_beta0}\min\Big\{\frac{1}{(2n_K)!\log q_{L/K}},\frac{1}{q_{L/K}^{1/n_K}}\Big\}.
\]
\end{lemma}

\begin{lemma}
\label{lem:ZFR_d-dim}
Let $L/K$ be a Galois extension of number fields with Galois group $G$ satisfying AHC.  There exists a constant $\Cl[abcon]{ZFR_new}$ such that if $\chi\in\Irr(G)$ and $\chi(1)\geq 2$, then
\[
L(\sigma+it,\chi,L/K)\neq 0,\qquad \sigma\geq 1-\frac{\Cr{ZFR_new}}{\log(q_{L/K}(|t|+3)^{d_G^2 n_K})}.
\]
\end{lemma}

\begin{proof}
Suppose that $\chi(1)\geq 2$ and that $\rho=\beta+i\gamma$ is a nontrivial zero of $L(s,\chi)$.  Express the tensor product $\chi \otimes \overline{\chi}$ as a direct sum of its irreducible constituents, as discussed in \cref{subsec:orthogonality}.  Per \cref{lem:tensor-contains-triv}(3), $\chi\otimes\bar{\chi}$ contains a constituent $\tau\in\Irr(G)-\{\mathbbm{1}_G,\chi,\bar{\chi}\}$ such that there exists $\psi\in\Rep(G)$ satisfying $\chi\otimes\tau = \chi\oplus\psi$ and $\langle \psi,\mathbbm{1}\rangle=0$.

Let
\begin{align*}
D(s) &= \zeta_K(s) L(s,\chi\otimes\bar{\chi},L/K) L(s,\tau\otimes\bar{\tau},L/K) L(s+i\gamma,\chi,L/K)L(s-i\gamma,\bar{\chi},L/K)\\
&\quad\times L(s+i\gamma,\chi\otimes\tau,L/K) L(s-i\gamma,\bar{\chi}\otimes\bar{\tau},L/K) L(s,\tau,L/K)L(s,\bar{\tau},L/K).
\end{align*}
Write
\[
-\frac{D'}{D}(s) = \sum_{\ka}\frac{a_D(\ka)\Lambda_K(\ka)}{\N\ka^{s}},\qquad\re(s)>1.
\]
It follows from \eqref{eqn:ap_formula} that
\begin{align*}
a_{D}(\kp^{\ell}) &= 1+a_{\chi\otimes\bar{\chi}}(\kp^{\ell})+a_{\tau\otimes\bar{\tau}}(\kp^{\ell})+a_{\chi}(\kp^{\ell})\N\kp^{-i\ell\gamma}+a_{\bar{\chi}}(\kp^{\ell})\N\kp^{i\ell\gamma}+a_{\chi\otimes\tau}(\kp^{\ell})\N\kp^{-i\ell\gamma}\\
&\qquad +a_{\bar{\chi}\otimes\bar{\tau}}(\kp^{\ell})\N\kp^{i\ell\gamma}+a_{\tau}(\kp^{\ell})+a_{\bar{\tau}}(\kp^{\ell})\\
&=\frac{1}{|I_{\kP}|}\sum_{\alpha\in I_{\kP}}(1+\chi(\varphi^\ell\alpha)\bar{\chi}(\varphi^\ell\alpha)+\tau(\varphi^\ell\alpha)\bar{\tau}(\varphi^\ell\alpha)+\chi(\varphi^\ell\alpha)\N\kp^{-i\ell \gamma}+\bar{\chi}(\varphi^\ell\alpha)\N\kp^{i\ell\gamma}\\
&\qquad\qquad\qquad+\chi(\varphi^\ell\alpha)\tau(\varphi^\ell\alpha)\N\kp^{-i\ell\gamma}+\bar{\chi}(\varphi^\ell\alpha)\bar{\tau}(\varphi^\ell\alpha)\N\kp^{i\ell\gamma}+\tau(\varphi^\ell\alpha)+\bar{\tau}(\varphi^\ell\alpha))\\
&=\frac{1}{|I_{\kP}|}\sum_{\alpha\in I_{\kP}}|1+\chi(\varphi^\ell\alpha)\N\kp^{-i\ell\gamma}+\bar{\tau}(\varphi^\ell\alpha)|^2\geq 0.
\end{align*}
Otherwise, $a_D(\ka)=0$. Keeping in mind our definition of $\psi$, we infer that if $\sigma>1$, then
\begin{align*}
0\leq -\frac{D'}{D}(\sigma) &=-\frac{\zeta_K'}{\zeta_K}(\sigma)- \frac{L'}{L}(\sigma, \chi\otimes\overline{\chi},L/K)- \frac{L'}{L}(\sigma, \tau\otimes\overline{\tau},L/K) - \frac{L'}{L}(\sigma + i\gamma, \chi,L/K) \\
    &\quad- \frac{L'}{L}(\sigma - i\gamma, \overline{\chi},L/K) - \frac{L'}{L}(\sigma + i\gamma, \chi\otimes\tau,L/K) - \frac{L'}{L}(\sigma - i\gamma, \overline{\chi}\otimes\overline{\tau},L/K) \\
        & \quad - \frac{L'}{L}(\sigma, \tau,L/K) - \frac{L'}{L}(\sigma, \overline{\tau},L/K)\\
&= -\frac{\zeta_K'}{\zeta_K}(\sigma)- \frac{L'}{L}(\sigma, \chi\otimes\overline{\chi},L/K)- \frac{L'}{L}(\sigma, \tau\otimes\overline{\tau},L/K) - 2\frac{L'}{L}(\sigma + i\gamma, \chi,L/K) \\
    &\quad- 2\frac{L'}{L}(\sigma - i\gamma, \overline{\chi},L/K) - \frac{L'}{L}(\sigma + i\gamma, \psi,L/K) - \frac{L'}{L}(\sigma - i\gamma, \overline{\psi},L/K) \\
        & \quad - \frac{L'}{L}(\sigma, \tau,L/K) - \frac{L'}{L}(\sigma, \overline{\tau},L/K).
\end{align*}
Rearranging the terms, we obtain
    \begin{equation}
    \label{eqn:real-part-inequality}
    \begin{aligned}
    	4\re\Big(\frac{L'}{L}(\sigma+i\gamma,\chi,L/K)\Big)\le -\re\Big(&\frac{\zeta_K'}{\zeta_K}(\sigma)+\frac{L'}{L}(\sigma, \tau\otimes\overline{\tau},L/K)+\frac{L'}{L}(\sigma, \chi\otimes\overline{\chi},L/K)\\
     &+2\frac{L'}{L}(\sigma,\tau,L/K)+2\frac{L'}{L}(\sigma+i\gamma,\psi,L/K)\Big).
    \end{aligned}
    \end{equation}
    
Note that if $\sigma>1$ and $t\in\R$, then
\[
\re\Big(\frac{1}{\sigma+it-\rho}\Big)=\frac{\sigma-\re(\rho)}{(\sigma-\re(\rho))^2+(t-\im(\rho))^2}\geq 0.
\]
Therefore, if $\beta+i\gamma$ is a nontrivial zero of $L(s,\chi,L/K)$, then by repeated applications of \eqref{eqn:L'/L-equality} and \cref{lem:Stirling}, there exists an absolute constant $\Cl[abcon]{4-3-inequality constant}\geq 1$ such that if $1<\sigma<2$ then
    \begin{equation}
    \label{eqn:real-part-inequality2}
    \begin{aligned}
    	\frac{4}{\sigma-\beta}<\frac{3}{\sigma-1}+\Cr{4-3-inequality constant}\log(D_K^{d_G^2}(|\gamma|+3)^{d_G^2 n_K}\N(\kf_{\chi}\kf_{\chi\otimes\bar{\chi}}\kf_{\tau}\kf_{\tau\otimes\bar{\tau}}\kf_{\psi})).
    \end{aligned}
    \end{equation}
Let $\kf_{L/K}$ be the Artin conductor of $\Gal(L/K)$ with greatest norm.  Since $\chi,\tau\in\Irr(G)$, the bounds $\N\kf_{\chi}\leq \N \kf_{L/K}$ and $\N\kf_{\tau}\leq \N\kf_{L/K}$ hold.  \cref{lem:tensor conductor bound} yields $\N \kf_{\chi\otimes\bar{\chi}}\leq \N\kf_{\chi}^{2d_G}\leq \N\kf_{L/K}^{2d_G}$ and $\N \kf_{\tau\otimes\bar{\tau}}\leq \N\kf_{\tau}^{2d_G}\leq \N\kf_{L/K}^{2d_G}$.  Finally, since $\psi$ is a sub-representation of $\chi\otimes\tau$, it follows from \cref{lem:tensor conductor bound} that $\N\kf_{\psi}\leq \N \kf_{\chi\otimes\tau}\leq (\N\kf_{\chi} \N\kf_{\tau})^{d_G} \leq \N\kf_{L/K}^{2d_G}$. Inserting these bounds into \eqref{eqn:real-part-inequality2}, we arrive at
\[
\frac{4}{\sigma-\beta}<\frac{3}{\sigma-1}+8\Cr{4-3-inequality constant}\log(q_{L/K}(|\gamma|+3)^{d_G^2 n_K}).
\]
Choosing
\[
\sigma=1+\frac{1}{16 \Cr{4-3-inequality constant}\log(q_{L/K}(|\gamma|+3)^{d_G^2 n_K})},
\]
we obtain a contradiction unless
\[
\beta\leq 1-\frac{1}{112\Cr{4-3-inequality constant}\log(q_{L/K}(|\gamma|+3)^{d_G^2 n_K})}.\qedhere
\]
\end{proof}

\begin{remark}
It follows from work of V.K. Murty \cite[Propositions 3.7 and 3.8]{VKM} that a version of \cref{lem:ZFR_d-dim} holds with $\log(q_{L/K}(|t|+3)^{d_G^2 n_K})$ replaced by $\log(q_{L/K}^{d_G^2}(|t|+3)^{d_G^3 n_K})$.
\end{remark}

\begin{proof}
This follows from \eqref{eqn:factorization of zeta_L} and \cref{lem:ZFR_1-dim,lem:ZFR_d-dim}.
\end{proof}

\subsection{Log-free zero density estimate}
Our proof of \cref{prop:main_AC_long} rests on \cref{thm:ZFR} and the following log-free zero density estimate.   We will prove this in \cref{sec:MVT,sec:LFZDE}.

\begin{theorem}
\label{thm:LFZDE}
Let $L/K$ be a Galois extension of number fields with Galois group $G$ satisfying AHC.  Let $\beta_1=\beta_{1,L/K}$ and $\chi_1=\chi_{1,L/K}$ be as in \cref{thm:ZFR}.  For $U>1$, define
\[
\nu(U) = \min\{1,(1-\beta_1)\log U\}.
\]
For $T\geq 1$, $\sigma\geq 0$, and $\chi\in\Irr(G)$, define
\[
N_{\chi}^*(\sigma,T)=\begin{cases}
\#\{\rho=\beta+i\gamma\neq\beta_1\colon L(\rho,\chi,L/K)=0,~\beta\geq\sigma,~|\gamma|\leq T\}&\mbox{if $\beta_1$ lies in \eqref{eqn:ZFR},}\\
\#\{\rho=\beta+i\gamma\colon L(\rho,\chi,L/K)=0,~\beta\geq\sigma,~|\gamma|\leq T\}&\mbox{otherwise.}
\end{cases}
\]
There exists a constant $\Cl[abcon]{LFZDE}\geq 2$ such that if $\sigma\geq 0$, then
\[
\sum_{\chi\in\Irr(G)}N_{\chi}^*(\sigma,T)\ll (\Log d_G)(\cQ_{L/K}T^{d_{G}^2 n_K})^{\Cr{LFZDE}(\Log d_G)(1-\sigma)}\nu(q_{L/K} T^{d_{G}^2n_K}).
\]
\end{theorem}

\section{Proof of \cref{thm:main_long} assuming \cref{prop:main_AC_long}}
\label{sec:reduction}

\subsection{Lower bound of the main term}

Recall the notation and hypotheses of \cref{thm:ZFR}.

\begin{lemma}
\label{lem:lower bound of main term}
Let $L/K$ be a Galois extension of number fields satisfying AHC.   If $A>1$ and $x\geq (2\cQ_{L/K})^{6A}$, then
\begin{equation}\label{eqn:lower bound of Li}
    \frac{|C|}{|G|}(\Li(x)-\chi_1(C)\Li(x^{\beta_1}))\gg  \frac{1}{A}x^{1-\frac{1}{A}}
\end{equation}
and
\[
x-\chi_1(C)\frac{x^{\beta_1}}{\beta_1}\gg\nu(\cQ^{\Log d_G})x \gg x^{1-\frac{1}{2A}}.
\]
\end{lemma}
\begin{proof}
This result is trivial unless $\beta_1\geq 1/2$.  For the first result, note that
\[
|G|\le d^2|\Irr(G)|\le \cQ_{L/K}\le \frac{1}{2}x^{\frac{1}{6A}}.
\]
Since $\chi_1(C)\in\{-1,1\}$, we have that
\[
\frac{|C|}{|G|}(\Li(x)-\chi_1(C)\Li(x^{\beta_1}))\geq \frac{|C|}{|G|}\int_{x^{\beta_1}}^x\frac{dt}{\log t}\ge \frac{x(1-x^{1-\beta_1})}{x^{1/(6A)}\log x}.
\]
It follows from the upper bound of $\beta_1$ in \cref{thm:ZFR} and our lower bound $x\geq 64$ that if $\beta_{1,L/K}\geq 1/2$, then
\[
1-x^{\beta_1-1}\geq 1-\frac{x^{\beta_1-1}}{\beta_1}\gg \min\{1,(1-\beta_1)\log x\} = \nu(x) \ge 1-\beta_1 \gg \cQ_{L/K}^{-3}\gg x^{-\frac{1}{2A}}.
\]
To finish, we apply the bound $\log x\leq \epsilon^{-1}x^{\epsilon}$, valid for $\epsilon>0$, with $\epsilon=1/(3A)$.
\end{proof}

\subsection{\cref{prop:main_AC_long} implies \cref{thm:main_long}}
We will use the following lemma.
\begin{lemma}[{\cite[Section 2.6]{Serre}}]
	\label{lem:MMS}
    Let $L/F$ be a Galois extension of number fields with Galois group $G$, and let $C\subseteq G$ be a conjugacy class.  Let $H$ be a subgroup of $G$ such that $C\cap H$ is nonempty, and let $K=L^H$ be the fixed field of $L$ by $H$.  Let $u\in C\cap H$, and let $C_H$ denote the conjugacy class of $H$ which contains $u$.  If $x\geq 2$, then
\[
\Big|\pi_C(x,L/F)-\frac{|C|}{|G|}\frac{|H|}{|C_H|}\pi_{C_H}(x,L/K)\Big|\leq\frac{|C|}{|G|}\Big(n_L x^{1/2}+\frac{2}{\log 2}\log D_L\Big).
\]
\end{lemma}

\cref{lem:MMS} allows us to base change to a sub-extension that satisfies AHC when the original extension does not. We also need the following relation between $\chi_{1,L/F}$ to $\chi_{1,L/L^H}$.

\begin{lemma}
\label{lem:restriction of chi_1}
	Let $L/F$ be a Galois extension of number fields with Galois group $G$. Suppose $H$ is a subgroup of $G$, and let $K=L^H$.  If $\chi_{1,L/F}$ and $\chi_{1,L/K}$ are as in \eqref{eqn:def of chi_1}, then $\chi_{1,L/K}=\chi_{1,L/F}|_H$.
\end{lemma}
\begin{proof}
Let $\beta_1$ be the greatest real zero of $\zeta_L(s)$, as in \eqref{eqn:def of beta_1}. If $\beta_1$ is not simple, then $\chi_{1,L/F}$ and $\chi_{1,L/K}$ are the trivial characters (of $G$ and $H$, respectively), so the claim holds. 

Suppose now that $\beta_1$ is simple. Let $\chi_{1,L/F}|_H$ be the restriction of $\chi_{1,L/F}$ to $H$.  \cref{prop:inductive invariance} implies that 
\begin{equation}
	\label{eqn:induction-well-defined}
	L(s,\chi_{1,L/F}|_H,L/L^H)=L(s,\operatorname{Ind}_H^G(\chi_{1,L/F}|_H),L/F).
\end{equation}
By Frobenius reciprocity, we have $\langle \operatorname{Ind}_H^G(\chi_{1,L/F}|_H),\,\chi_{1,L/F}\rangle=\langle \chi_{1,L/F}|_H,\,\chi_{1,L/F}|_H\rangle=1$.  That is, $\operatorname{Ind}_H^G(\chi_{1,L/F}|_H)$ contains $\chi_{1,L/F}$ as an irreducible constituent. Hence, $\beta_1$ is a zero of $L(s,\operatorname{Ind}_H^G(\chi_{1,L/F}|_H),L/F)$. By \eqref{eqn:induction-well-defined} we know that $\beta_1$ is also a zero of $L(s,\chi_{1,L/F}|_H,L/L^H)$. Since $\chi_{1,L/F}|_H$ is a real quadratic character of $H$ and the exceptional character of $H$ is unique, we conclude that $\chi_{1,L/F}|_H=\chi_{1,L/K}$. This establishes the claim.
\end{proof}

In particular, if $C$ is a conjugacy class of $G$, $H$ is a subgroup of $G$, and $C_H$ is the conjugacy class of $H$ defined as in \cref{lem:MMS}, then $\chi_{1,L/F}(C)=\chi_{1,L/K}(C_H)$.

\begin{proof}[Proof of \cref{thm:main_long}, assuming \cref{prop:main_AC_long}]
Let $H\in\AHC(G)$ and $K=L^H$.  If $C_H$ is as in \cref{lem:MMS}, then
$$
\pi_{C_H}(x,L/K)=\frac{|C_H|}{|H|}(\Li(x)-\chi_{1,L/K}(C_H))\Li(x^{\beta_1}))(1+\delta_{C_H}(x,L/K)).
$$
\cref{lem:MMS} implies that
\[
\pi_C(x,L/F)=\frac{|C|}{|G|}(\Li(x)-\chi_{1,L/K}(C_H)\Li(x^{\beta_{1,L/K}}))(1+\delta_{C_H}(x,L/K))+O(n_L x^{1/2}+\log D_L),
\]
while
\[
\pi_C(x,L/F)=\frac{|C|}{|G|}(\Li(x)-\chi_{1,L/F}(C)\Li(x^{\beta_{1,L/F}}))(1+\delta_C(x,L/F))
\]
by the definition of $\delta_C(x,L/F)$. Since $\beta_{1,L/F}=\beta_{1,L/K}$ (both of which are the greatest real zero of $\zeta_L(s)$) and $\chi_{1,L/F}(C)=\chi_{1,L/K}(C_H)$, it follows that
\[
|\delta_C(x,L/F)|= |\delta_{C_H}(x,L/K)|+O\Big(\frac{n_L x^{1/2}+\log D_L}{\frac{|C|}{|G|}(\Li(x)-\Li(x^{\beta_1}))}\Big).
\]

From the bound $n_L\leq 5\log D_L\le 5\cQ_{L/K}^2$ and the assumptions in \cref{thm:main_long}, we see that
\[
n_Lx^{1/2}+\log D_L\le 6\cQ_{L/K}^2x^{1/2}\le 6x^{3/4}.
\]
Using the bound in \eqref{eqn:lower bound of Li}, we conclude that $|\delta_C(x,L/F)|=|\delta_{C_H}(x,L/K)|+O(x^{-\frac{1}{10}})$.  Applying \cref{prop:main_AC_long} to the extension $L/K$, we find that there exist constants $\Cr{thm:main1long}\geq \Cr{cor:main1}$, $\Cr{thm:main2long}\geq \Cr{cor:main2}$, and $\Cr{thm:main3long}\leq \Cr{cor:main3}$ such that if $x\geq (2\cQ_{L/K})^{\Cr{thm:main1long}\Log d_G}$, then $|\delta_C(x,L/F)|$ is
\begin{align*}
&\le \Cr{cor:main2} \exp\Big(\frac{-\Cr{cor:main3}\log x}{\log q_{L/K}+\sqrt{d_H^2 n_K\log x}}\Big)+O(x^{-\frac{1}{10}})\\
&\leq \Cr{thm:main2long}\exp\Big(\frac{-\Cr{thm:main3long}\log x}{ \log q_{L/K}+\sqrt{d_H^2 n_K \log x}}\Big)<1.\qedhere
\end{align*}
\end{proof}

\section{Proof of \cref{prop:main_AC_long}}
\label{sec:proof_main}

For the rest of this paper, we will work with a Galois extension $L/K$ of number fields with Galois group $G$ satisfying AHC. Let $\cO_K$ be the ring of integers of $K$, and let $\kp$ (resp. $\ka$) denote a nonzero prime ideal (resp. nonzero ideal) of $\cO_K$.  Let $C\subseteq G$ be a conjugacy class.  Write $L(s,\chi)=L(s,\chi,L/K)$ and
\[
d=d_G,\qquad q=q_{L/K},\qquad \cQ=\cQ_{L/K},\qquad \beta_1=\beta_{1,L/K},\qquad \chi_1=\chi_{1,L/K}.
\]
In what follows, we will assume that $\beta_1$ lies in the region \eqref{eqn:ZFR}.  If not, then our proofs become easier.  With \cref{thm:ZFR,thm:LFZDE} in hand, the arguments in this section are nearly identical to those in \cite[Sections~2,~4,~and~5]{TZ3}.  For brevity, we will mostly provide sketches, referencing the corresponding results in \cite{TZ3}.

\subsection{Smoothing}

To ameliorate our analysis, we introduce the following smooth weight.

\begin{lemma}[{\cite[Lemma 2.2]{TZ3}}]
\label{lem:WeightChoice}
Let $x \geq 3$, $\epsilon \in (0,\frac{1}{4})$, and an integer $\ell \geq 2$. Define $A = \epsilon/(2 \ell \log x)$. There exists a continuous function $f(t)  = f(t; x, \ell, \epsilon)$ of a real variable $t$ such that the following statements are true.
\begin{enumerate}[(i)]
\item If $t\in\R$, then $0 \leq f(t) \leq 1$.  Also, $f(t)= 1$ for $\tfrac{1}{2} \leq t \leq 1$.
\item The support of $f$ is contained in the interval $[\tfrac{1}{2}-\frac{\epsilon}{\log x}, 1+\frac{\epsilon}{\log x}]$.
\item The Laplace transform $F(z) = \int_{\R} f(t) e^{-zt} dt$ is entire and given by
\[
F(z) = e^{-(1+ 2\ell A)z} \cdot \Big(\frac{1-e^{(\frac{1}{2}+2\ell A)z}}{-z} \Big) \Big(\frac{1-e^{2Az}}{-2Az} \Big)^{\ell}.
\]
\item If $s = \sigma + it$, $\sigma > 0$, $t \in \R$, and $\alpha\in[0,\ell]$, then
\[
|F(-s \log x)| \leq \frac{e^{\sigma \epsilon} x^{\sigma}}{|s| \log x} \cdot 
(1+x^{-\sigma/2} ) \cdot \Big(\frac{2\ell}{\epsilon |s|} \Big)^{\alpha}.
\]
Moreover, $|F(-s \log x)| \leq e^{\sigma \epsilon} x^{\sigma}$ and $1/2 < F(0) < 3/4$.
\item If $\frac{3}{4}<\sigma\leq 1$ and $x\geq 10$, then
\[
F(-\log x) \pm F(-\sigma\log x)=\Big(\frac{x}{\log x}\pm\frac{x^{\sigma}}{\sigma \log x}\Big)(1+O(\epsilon))+O\Big(\frac{x^{1/2}}{\log x}\Big).
\]
\item If $t\in\R$ and $s=-\frac{1}{2}+it$, then
\[
|F(-s\log x)|\leq\frac{5x^{-1/4}}{\log x}\Big(\frac{2\ell}{\epsilon}\Big)^{\ell}(\tfrac{1}{4}+t^2)^{-\ell/2}.
\]
\end{enumerate}
\end{lemma}

Let $f$ be as in \cref{lem:WeightChoice}, with the choices
\begin{equation}
\label{eqn:x_ell_epsilon_def}
\ell = \lceil 4\Cr{LFZDE}d^2 n_K \Log d\rceil \geq 2,\qquad x\geq (2\cQ)^{36\Cr{LFZDE}\Log d}, \qquad \epsilon=8\ell x^{-\frac{1}{8\ell}}\in(0,\tfrac{1}{4}).
\end{equation}
Note that \eqref{eqn:degree_q_def} and \eqref{eqn:conductor discriminant formula} imply that
\begin{equation}
\label{eqn:log DL << Q^2}
n_L\leq 5 \log D_L = 5 \Big(|G|\log D_K+\sum_{\chi\in\Irr(G)}\chi(1)\log\N\kf_{\chi}\Big)\leq \cQ^2.
\end{equation}
We also record the convenient results
\begin{equation}
\label{eqn:char_theory_bounds}
|\chi(C)|\leq \chi(1),\qquad \sum_{\chi\in\Irr(G)}|\chi(C)|^2 = \frac{|G|}{|C|}.
\end{equation}

Recall the definition of $\tau_C$ in \eqref{eqn:tau_def}. Define
\begin{align*}
\Psi_C(x,f) &= \sum_{\ka}\Lambda_K(\ka)\tau_C(\ka)f\Big(\frac{\log\N\ka}{\log x}\Big) \\
&= -\frac{|C|}{|G|}\sum_{\chi}\bar{\chi}(C)\frac{\log x}{2\pi i}\int_{2-i\infty}^{2+i\infty}\frac{L'}{L}(s,\chi)F(-s\log x)ds+O((\log x)(\log D_L)).
\end{align*}
The second line follows via Mellin inversion and an estimation of the contribution from the $\ka$ that are powers of ramified prime ideals.
\begin{lemma}
\label{lem:smoothing}
Let $f$ be as in \cref{lem:WeightChoice}.  If $\ell$, $\epsilon$, and $x$ satisfy \eqref{eqn:x_ell_epsilon_def}, then
\[
\pi_C(x) = \frac{\Psi_C(x;f)}{\log x}+\int_{\sqrt{x}}^{x}\frac{\Psi_C(t;f)}{t(\log t)^2}dt + O\Big(\cQ^2+\frac{n_K \sqrt{x}}{\log x}+\epsilon x\Big).
\]
\end{lemma}
\begin{proof}
Under our hypotheses, the asymptotic
\[
\pi_C(x) = \frac{\Psi_C(x;f)}{\log x}+\int_{\sqrt{x}}^{x}\frac{\Psi_C(t;f)}{t(\log t)^2}dt + O\Big((\log D_L)(\log x)+\frac{n_K \sqrt{x}}{\log x}+\epsilon x\Big)
\]
follows from \cref{lem:WeightChoice} and \cite[Lemmata 2.1 and 2.3]{TZ3}.   To finish, we apply \eqref{eqn:log DL << Q^2}.
\end{proof}

\subsection{An explicit formula}

Let $x$, $\ell$, and $\epsilon$ be as in \eqref{eqn:x_ell_epsilon_def}.  Recall that we are assuming that $\beta_1>\frac{1}{2}$; therefore, upon pushing the contour to $\re(s)=-\frac{1}{2}$, we have that
\begin{align*}
\frac{|G|}{|C|}\Psi_C(x;f)&=(F(-\log x)-\chi_1(C)F(-\beta_1\log x))\log x\\
&-\sum_{\chi\in\Irr(G)}\bar{\chi}(C)\sum_{\substack{\rho\neq\beta_1 \\ \Lambda(\rho,\chi)=0}}F(-\rho\log x)\log x\\
&-\frac{\log x}{2\pi i }\sum_{\chi}\bar{\chi}(C)\int_{-\frac{1}{2}-i\infty}^{-\frac{1}{2}+i\infty}\frac{L'}{L}(s,\chi)F(-s\log x)ds.
\end{align*}
By \cref{lem:WeightChoice}(v), we have that
\[
(F(-\log x)-\chi_1(C)F(-\beta_1\log x))\log x=\Big(x-\chi_1(C)\frac{x^{\beta_1}}{\beta_1}\Big)(1+O(\epsilon))+O(x^{1/2}).
\]
By standard calculations using Lemmata \ref{lem:Riemann-vonMangoldt}, \ref{lem:lower bound of main term}, and \ref{lem:WeightChoice} as well as \eqref{eqn:x_ell_epsilon_def}-\eqref{eqn:char_theory_bounds} that are very similar to \cite[Lemmata 4.3 and 4.4]{TZ3}, we arrive at
\begin{equation}
\label{eqn:until_the_zeros}
\frac{|G|}{|C|}\Psi_C(x;f) = \Big(x-\chi_1(C)\frac{x^{\beta_1}}{\beta_1}\Big)(1+O(\epsilon))+O\Big(d\sum_{\chi\in\Irr(G)}\sum_{\substack{\rho\neq\beta_1 \\\Lambda(\rho,\chi)=0 \\  |\rho|>\frac{1}{4}}}|F(-\rho\log x)|\log x\Big).
\end{equation}

\begin{lemma}
\label{lem:sum_over_zeros_high}
Recall the notation and hypotheses of \cref{thm:ZFR}, \cref{thm:LFZDE}, and \eqref{eqn:x_ell_epsilon_def}.  If $\eta(x) = \inf_{t\geq 3}[\Delta(t)\log x+\log t]$, then
\[
\sum_{\chi\in\Irr(G)}\sum_{\substack{\rho\neq \beta_1 \\ \Lambda(\rho,\chi) =0 \\ |\rho|>\frac{1}{4}}}|F(-\rho\log x)|\log x\ll (\Log d)\nu(qT^{d^2 n_K})x e^{-\eta(x)/2}
\]
and
\[
e^{-\eta(x)}\leq \exp\Big(-\frac{\Cr{ZFR}}{2}\cdot\frac{\log x}{\log q+\sqrt{d^2 n_K\log x}}\Big)
\]
\end{lemma}
\begin{proof}
This is proved using \cref{thm:ZFR,thm:LFZDE}.  The proof is essentially the same as that of \cite[Lemmata 4.5 and 4.6]{TZ3}.
\end{proof}

\begin{proof}[Proof of \cref{prop:main_AC_long}]
Per \eqref{eqn:x_ell_epsilon_def}, \eqref{eqn:until_the_zeros}, and \cref{lem:sum_over_zeros_high}, we find that
\begin{align*}
\frac{|G|}{|C|}\Psi_C(x;f) &= \Big(x-\chi_1(C)\frac{x^{\beta_1}}{\beta_1}\Big)(1+O(d^2 n_K(\Log d) x^{-\frac{1}{32d^2 n_K\Log d}}))\\
&+O\Big(d(\Log d)\nu(qT^{d^2 n_K})x\exp\Big(-\frac{\Cr{ZFR}}{2}\cdot\frac{\log x}{\log q+\sqrt{d^2 n_K\log x}}\Big)\Big).
\end{align*}

It now follows from \cref{lem:lower bound of main term} that
\[
\Psi_C(x;f) = \frac{|C|}{|G|}\Big(x-\chi_1(C)\frac{x^{\beta_1}}{\beta_1}\Big)\Big(1+O\Big(d(\Log d)\exp\Big(-\frac{\Cr{ZFR}}{2}\cdot\frac{\log x}{\log q+\sqrt{d^2 n_K\log x}}\Big)\Big)\Big).
\]
In light of the range of $x$ in \eqref{eqn:x_ell_epsilon_def}, we conclude that
\begin{equation}
\label{eqn:last_step_explicit_formula}
\Psi_C(x;f) = \frac{|C|}{|G|}\Big(x-\chi_1(C)\frac{x^{\beta_1}}{\beta_1}\Big)\Big(1+O\Big(\exp\Big(-\frac{\Cr{ZFR}}{4}\cdot\frac{\log x}{\log q+\sqrt{d^2 n_K\log x}}\Big)\Big)\Big).
\end{equation}
It remains to insert \eqref{eqn:last_step_explicit_formula} into the estimate in \cref{lem:smoothing} and incorporate \eqref{eqn:x_ell_epsilon_def}.  The calculations are similar to those in \cite[Section 5A]{TZ3}.
\end{proof}

\section{A mean value estimate for Artin $L$-functions}
\label{sec:MVT}

When $d=1$, we can deduce a log-free zero density estimate of the form in \cref{thm:LFZDE} (with an improved exponent) from the log-free zero density estimate and zero repulsion result in \cite[Theorems 4.4 and 4.5]{TZ2}, arguing as in the proof of \cite[Theorem 2.1]{TZ_PNTAP}.  Therefore, we restrict our our consideration to $d\geq 2$, where the novelty lies.  This  will also be convenient for our proofs.  Our proof of \cref{thm:LFZDE} relies on the following mean value estimate.

\begin{proposition}
\label{prop:hybrid_MVT}
Let $b$ be a complex-valued function on the nonzero prime ideals of $\cO_K$ such that $\sum_{\kp}|b(\kp)|<\infty$.  There exists a constant $\Cl[abcon]{lower00}>426$ such that if $T\geq 1$ and
\begin{equation}
	\label{eqn:def of M and delta}
	y\geq \cQ^{\Cr{lower00}\Log d}T^{20d^2 n_K},
\end{equation}
then
\begin{align*}
\sum_{\chi\in\Irr(G)}\int_{-T}^{T}\Big|\sum_{\N\kp>y}\frac{a_{\chi}(\kp)b(\kp)}{\N\kp^{it}}\Big|^2 dt\ll \frac{\Log d}{\log y}\sum_{\N\kp>y}|b(\kp)|^2 \N\kp.
\end{align*}
\end{proposition}

\subsection{Schur polynomials and Dirichlet coefficients}
\label{subsec:Schur}

A partition $\mu=(\mu_i)_{i=1}^{\infty}$ is a sequence of nonincreasing nonnegative integers $\mu_1\geq \mu_2\geq\cdots$ with finitely many nonzero terms.  For a partition $\mu$, let $\ell(\mu)=\#\{i\colon \mu_i\neq 0\}$ and $|\mu|=\sum_{i=1}^{\infty}\mu_i$.  For an unordered list of complex numbers $\{\alpha_1,\ldots,\alpha_n\}\subseteq\mathbb{C}$ and a partition $\mu$ with $\ell(\mu)\leq n$, define the Schur polynomial associated to $\mu$ by
\[
s_{\mu}(\{\alpha_1,\ldots,\alpha_n\})=\det[(\alpha_i^{\mu(j)+n-j})_{ij}]/\det[(\alpha_i^{n-j})_{ij}].
\]
By convention, if $|\mu|=0$, then $s_{\mu}(\{\alpha_1,\ldots,\alpha_n\})$ is identically $1$.  Also, if $\ell(\mu)>n$, then $s_{\mu}(\{\alpha_1,\ldots,\alpha_n\})$ is identically $0$.  If $\{\alpha_1,\ldots,\alpha_m\}$ and $\{\beta_1,\ldots,\beta_n\}$ are sets of complex numbers of modulus less than $1$, then Cauchy's identity \cite[(38.1)]{Bump} states
\begin{equation}
\label{eqn:Cauchy_identity}
\prod_{i=1}^m \prod_{j=1}^n \frac{1}{1-\alpha_i \beta_j}=\sum_{\mu}s_{\mu}(\{\alpha_1,\ldots,\alpha_m\})s_{\mu}(\{\beta_1,\ldots,\beta_n\}).
\end{equation}

Let $L/K$ be a Galois extension of number fields with Galois group $G$. Recall the notation in \eqref{eqn:set of local roots} and the definitions in \eqref{eqn:eigenvalue definition of Artin $L$-functions}. Let $\chi\in\Irr(G)$. Note that $s_\mu(\{\beta\})=0$ unless $\mu$ is of the form $\mu=(k,0,0\ldots)$, in which case $s_{(k,0,0\ldots)}(\{\beta\})=\beta^{k}$. Hence, \eqref{eqn:Cauchy_identity} implies that, for any prime ideal $\kp$ in $\OK$, we have
\begin{equation}\label{eqn:Dirichlet series of Euler factor at p}
	L_{\kp}(s,\chi) = \sum_{k=0}^{\infty}\frac{s_{(k,0,0,\ldots)}(A_{\chi}(\kp))}{\N\kp^{ks}}=1+\sum_{k=1}^{\infty}\frac{s_{(k,0,0,\ldots)}(A_{\chi}(\kp))}{\N\kp^{ks}},\quad \Re(s)>1. 
\end{equation}

Now, let $\chi,\chi'\in\Irr(G)$, and write $V,V'$ for the associated vector spaces. If $\kp$ is a prime ideal in $\OK$ such that $\kp\nmid \kf_{\chi}\kf_{\chi'}$, then by \cref{lem:p_divides_conductor} we know that $V^{I_\kP}=V$ and $(V')^{I_\kP}=V'$. In particular, $(V\otimes V')^{I_\kP}=V\otimes V'$. This implies that
\[
A_{\chi\otimes\chi'}(\kp)=\{\alpha_{j,\chi}(\kp)\alpha_{j',\chi'}(\kp)\colon 1\le j\le \chi(1),1\le j'\le \chi'(1)\},
\]
as an (unordered) list of complex numbers of modulus 1. Hence, when $\kp\nmid \f_\chi\f_{\chi'}$, \eqref{eqn:Cauchy_identity} implies that
\begin{equation}\label{eqn:tensor product of Dirichlet series of Euler factor at p}
	L_{\kp}(s,\chi\otimes\chi') = \sum_{\mu}\frac{s_{\mu}(A_{\chi}(\kp))s_{\mu}(A_{\chi'}(\kp))}{\N\kp^{|\mu|s}},\quad \Re(s)>1.
\end{equation}

For a nonzero ideal $\ka$ of $\cO_K$, let $\lambda_{\chi}(\ka)$ (resp. $\lambda_{\chi\otimes\chi'}(\ka)$) be the $\ka$-th Dirichlet coefficient of $L(s,\chi)$ (resp. $L(s,\chi\otimes\chi')$).  Let
\[
\ka = \prod_{\kp}\kp^{\mathrm{ord}_{\kp}(\ka)}
\]
be the prime ideal factorization of $\ka$.  Let $(\mu_{\kp})_{\kp}$ denote a sequence of partitions indexed by the prime ideals of $\cO_K$, and define
\[
\underline{\mu}[\ka]=\{(\mu_{\kp})_{\kp}\colon \textup{$|\mu_{\kp}|=\mathrm{ord}_{\kp}(\ka)$ for all $\kp$}\}.
\]
It follows from \eqref{eqn:Dirichlet series of Euler factor at p} that
\[
\lambda_{\chi}(\ka)=\prod_{\kp}s_{(\mathrm{ord}_{\kp}(\ka),0,0,\ldots)}(A_{\chi}(\kp)).
\]
Also, if $\gcd(\ka,\kf_{\chi}\kf_{\chi'})=\cO_K$, then it follows from \eqref{eqn:tensor product of Dirichlet series of Euler factor at p} that
\begin{equation}
\label{eqn:schur}
\lambda_{\chi\otimes\chi'}(\ka) = \sum_{(\mu_{\kp})_{\kp}\in\underline{\mu}[\ka]}\prod_{\kp}s_{\mu_{\kp}}(A_{\chi}(\kp))s_{\mu_{\kp}}(A_{\chi'}(\kp)).
\end{equation}

\subsection{Selberg sieve weights}

Let $z\geq 2$. 
 Given $\chi,\chi'\in\Irr(G)$ and $\kd$ a nonzero ideal in $\OK$, let
\[
g_{\kd}(s,\chi\otimes\chi') = \prod_{\kp\mid\kd}\Big(1-\prod_{j=1}^{\chi(1)}\prod_{j'=1}^{\chi'(1)}(1-\alpha_{j,\chi}(\kp)\alpha_{j',\chi'}(\kp)\N\kp^{-s})\Big).
\]
Let $P^-(\ka)=\min\{\N\kp\colon \kp\mid\ka\}$, with the convention that $P^-(\cO_K)=\infty$.  For each $\chi\in\Irr(G)$, define
\begin{equation}
\label{eqn:definitions of g, P, D}
	g_{\chi}(\kd)=g_{\kd}(1,\chi\otimes\bar{\chi}),\qquad P_{\chi}(z):=\prod_{\substack{\N\kp<z  \\ g_{\chi}(\kp)\neq 0}}\kp,\qquad \cD_{\chi}(z)=\{\kd\colon \N\kd\leq z,~\kd\mid P_{\chi}(z)\}.
\end{equation}
Let $\rho_{\chi}(\kd)$ be a real-valued function satisfying
\begin{equation}
\label{eqn:Selberg_weights}
\rho_{\chi}(\cO_F)=1,\qquad \rho_{\chi}(\kd)=0\textup{ unless $\kd\in\cD_{\chi}(z)$},\qquad |\rho_{\chi}(\kd)|\leq 1\textup{ for all $\kd$.}
\end{equation}
The definitions and hypotheses in \eqref{eqn:definitions of g, P, D} and \eqref{eqn:Selberg_weights} ensure that if $P^-(\ka)>z$, then the condition $\kd\mid\ka$ implies that either $\ka=\cO_K$ or $\rho_{\chi}(\kd)=0$.

\subsection{Partial sums}
\label{subsec:partialsums}

Let $\Phi\colon\R\to[0,\infty)$ be an infinitely differentiable function whose support is a compact subset of $(-2,2)$. Define the transform
\[
\hat{\Phi}(s) = \int_{0}^{\infty}\Phi(\log t)t^{s-1}dt.
\]
By Fourier inversion, we have, for $c\in\R$ and $T\geq 1$, the equality
  \begin{equation}\label{eqn:inversion formula}
    \Phi(T\log x)=\frac{1}{2\pi i T}\int_{c-i\infty}^{c+i\infty}\hat{\Phi}(s/T)x^{-s}ds.
  \end{equation}
For particular choices of $\Phi$, we can exhibit strong and explicit control over the growth of its derivatives.
\begin{lemma}
\label{lem:smoothlargesieve}
Define $\Phi_1\colon\R\to[0,2]$ and $\Phi_2\colon\R\to[0,1]$ by
\[
\Phi_1(t)=\begin{cases}
4\exp\big(\frac{1}{(t-1/2)^2-1}\big)&\mbox{if $-\frac{1}{2}<t<\frac{3}{2}$,}\\
0 & \mbox{otherwise,}
\end{cases}\quad 
\Phi_2(t)=\begin{cases}
\exp\big(\frac{1}{4(t-1/2)^2-1}\big) & \mbox{if $0<t<1$,}
\\
0 & \mbox{otherwise.}
\end{cases}
\]
If $m\geq 1$ is an integer, then $\hat{\Phi}_1(s),\hat{\Phi}_2(s)\ll m^{2m}|s|^{-m}e^{2|\re(s)|}$.
\end{lemma}
\begin{proof}
This follows from integration by parts and \cite[Lemma 9]{BFI}.
\end{proof}

We note that if $\mathbbm{1}_{[0,1]}$ is the indicator function of $[0,1]$, then $\Phi_2(t)\le \mathbbm{1}_{[0,1]}(t)\le \Phi_1(t)$ for all $t\in\R$.  Let $\chi,\chi'\in\Irr(G)$, and write
\[
D(s;\chi,\chi') =L(s,\chi\otimes\bar{\chi}')\Big(\prod_{\kp\mid\kf_{\chi}\kf_{\chi'}}\prod_{j=1}^{\chi(1)}\prod_{j'=1}^{\chi'(1)}\frac{1-\alpha_{j,j',\chi\otimes\bar{\chi}'}(\kp)\N\kp^{-s}}{1-\alpha_{j,\chi}(\kp)\alpha_{j',\bar{\chi}'}(\kp)\N\kp^{-s}}\Big)g_{[\kd,\kd']}(s,\chi\otimes\bar{\chi}').
\]
If $j\in\{1,2\}$, $\kd\in\cD_{\chi}(z)$, and $\kd'\in\cD_{\chi'}(z)$, then by Mellin inversion, we have that
\begin{equation}
\label{eqn:integranded}
  \begin{aligned}
 &\sum_{[\kd,\kd']\mid \ka}\sum_{(\mu_{\kp})_{\kp}\in\underline{\mu}[\ka]}\Big(\prod_{\kp}s_{\mu_{\kp}}(A_{\chi}(\kp))\overline{s_{\mu_{\kp}}(A_{\chi'}(\kp))}\Big)\Phi_j\Big(T\log\frac{\N\ka}{x}\Big)\\
 &=\frac{1}{2\pi i T}\int_{3-i\infty}^{3+i\infty}D(s;\chi,\chi')\hat\Phi_j(s/T)x^s ds.
  \end{aligned}
  \end{equation}

Observe that if $\chi=\chi'$, then
\begin{align*}
\prod_{\kp\mid\kf_{\chi}}\prod_{j=1}^{\chi(1)}\prod_{j'=1}^{\chi(1)}\frac{1-\alpha_{j,j',\chi\otimes\bar{\chi}}(\kp)\N\kp^{-1}}{1-\alpha_{j,\chi}(\kp)\alpha_{j',\bar{\chi}}(\kp)\N\kp^{-1}}
=\prod_{\kp \mid \f_\chi} \frac{\sum_\mu \frac{\mid s_\mu\left(A_\chi(\kp))\right|^2}{\N\kp^{|\mu|}}}{\exp (\sum_{j=0}^{\infty} \frac{a_{\chi \otimes \bar{\chi}}\left(\kp^j\right)}{j \N\kp^j})}>0.
\end{align*}
In light of the identity
\[
L(s,\chi\otimes\bar{\chi}) = \exp\Big(\sum_{\N\ka\geq 2}\frac{a_{\chi\otimes\bar{\chi}}(\ka)\Lambda_K(\ka)}{\N\ka^s \log\N\ka}\Big),
\]
the lower bound $\mathrm{Res}_{s=1}L(s,\chi\otimes\bar{\chi})>0$ follows from \cref{lem:nonnegativity of dirichlet coefficients} and the existence of a simple pole at $s=1$.  Hence,
\begin{equation}
\label{eqn:kappa_def}
\kappa_{\chi,\bar{\chi}'}=\mathop{\mathrm{Res}}_{s=1}L(s,\chi\otimes\bar{\chi}')\prod_{\kp\mid\kf_{\chi}\kf_{\chi'}}\prod_{j=1}^{\chi(1)}\prod_{j'=1}^{\chi'(1)}\frac{1-\alpha_{j,j',\chi\otimes\bar{\chi}'}(\kp)\N\kp^{-1}}{1-\alpha_{j,\chi}(\kp)\alpha_{j'\bar{\chi}'}(\kp)\N\kp^{-1}}\begin{cases}
>0&\mbox{if $\chi=\chi'$,}\\
=0&\mbox{otherwise.}	
\end{cases}
\end{equation}

At $s=1$, the integrand of \eqref{eqn:integranded} has residue
\begin{equation}
\label{eqn:residued}
\begin{aligned}
\frac{x}{T}\hat{\Phi}_j(1/T)\kappa_{\chi,\bar{\chi}'}g_{\chi}([\kd,\kd']).
\end{aligned}
\end{equation}
We shift the contour to $\re(s) = 1-\frac{1}{2\Log d}$, so that \eqref{eqn:integranded} equals
\begin{equation}
\label{eqn:residued2}
\begin{aligned}
  \frac{x}{T}\hat{\Phi}_j(1/T)\kappa_{\chi,\bar{\chi}'}g_{\chi}([\kd,\kd'])+\frac{1}{2\pi i T}\int_{1-\frac{1}{2\Log d}-i\infty}^{1-\frac{1}{2\Log d}+i\infty}D(s;\chi,\chi')\hat{\Phi}_j(s/T)x^s ds.
  \end{aligned}
  \end{equation}
To estimate \eqref{eqn:residued2} we require an upper bound on the number of prime ideal divisors of a given ideal in $\OK$.
\begin{lemma}
\label{lem:prime_divisors}
  Let $\ka$ be a nonzero ideal of $\cO_K$, and write $\omega(\ka)$ for the number of distinct prime ideals in $\OK$ dividing $\ka$. Then for every $\delta>0$, we have $\omega(\ka)\le e^{1+1/\delta}n_K+\delta\log\N\ka$.
\end{lemma}
\begin{proof}
 We may assume that $\ka$ is squarefree and $\omega(\ka)\ge e^{1+1/\delta}n_K$.  For an integer $i\geq 1$, let $p_i$ be the $i$-th rational prime.  The function $\omega(\ka)$ is largest if for each $i$, there are $n_K$ distinct prime ideals in the factorization of $\ka$ lying above $p_i$, say $\kp_i^{(1)},\ldots,\kp_i^{(n_K)}$, each of which has norm $p_i$.
  
  First, suppose $\ka$ is of the form $\ka=(\kp_1^{(1)}\cdots\kp_1^{(n_K)})\cdots(\kp_m^{(1)}\cdots\kp_m^{(n_K)})$, in which case $\omega(\ka)=n_Km$.
  Then
  \[
  \log\N\ka= n_K\sum_{j=1}^m\log p_j\ge n_K\sum_{j=1}^m\log j\ge n_Km(\log m-1)\ge \frac{n_K m}{\delta}.
  \]
  Therefore, $\omega(\ka)\le \delta\log\N\ka$ in this case.
  
  Otherwise, suppose $\ka$ is of the form
  \[
  \ka=(\kp_1^{(1)}\cdots\kp_1^{(n_K)})\cdots(\kp_m^{(1)}\cdots\kp_m^{(n_K)})(\kp_{m+1}^{(1)}\cdots\kp_{m+1}^{(\ell)}),
  \]
  where $\ell<n_K$ and $m=\left\lfloor \omega(\ka)/n_K\right\rfloor$. Then we may apply the previous case to $\ka$ with all the factors $\kp_{m+1}^{(j)}$ removed. Therefore, we obtain $
  \omega(\ka)\le\delta\log\N\ka+n_K$. Since we assumed that $\omega(\ka)\ge e^{1+1/\delta}n_K$, we conclude that $\omega(\ka)\le e^{1+1/\delta}n_K+\delta\log\N\ka$ for all $\ka$.
\end{proof}

In what follows, we define
\[
\delta = \frac{1}{2\Log d}.
\]
\begin{lemma}
\label{lem:finite_Euler}
Let $\kd\in\cD_\chi(z),\kd'\in\cD_{\chi'}(z)$. 
 If $\re(s)=1-\delta$, then
\begin{align*}
|g_{[\kd,\kd']}(s,\chi\otimes\bar{\chi}')|&\ll z^{2e\delta}e^{30d^2 n_K}(2/\delta)^{ed^2 n_K},\\
\Big|\prod_{\kp\mid\kf_{\chi}\kf_{\chi'}}\prod_{j=1}^{\chi(1)}\prod_{j'=1}^{\chi'(1)}\frac{1-\alpha_{j,j',\chi\otimes\bar{\chi}'}(\kp)\N\kp^{-s}}{1-\alpha_{j,\chi}(\kp)\alpha_{j'\bar{\chi}'}(\kp)\N\kp^{-s}}\Big|&\ll q^{4e\delta}e^{93d^2 n_K}(2/\delta)^{4e d^2 n_K}.
\end{align*}
\end{lemma}

\begin{proof}
Let $\re(s) = 1-\delta$.  By the definition of $g_{[\kd,\kd']}(s,\chi\otimes\bar{\chi}')$ and the fact that Frobenius eigenvalues have modulus $1$, we find that
\[
|g_{[\kd,\kd']}(s,\chi\otimes\bar{\chi}')|\leq \prod_{\kp\mid[\kd,\kd']}\Big(1+\frac{1}{\N\kp^{1-\delta}}\Big)^{d^2}.
\]
We observe that
\[
\Big(1+\frac{1}{\N\kp^{1-\delta}}\Big)^{d^2}\leq \begin{cases}
e^e &\mbox{if $\N\kp>d^2$,}\\
\exp(e d^2/\N\kp)&\mbox{if $\N\kp\leq d^2$.}
\end{cases}
\]
It is straightforward to prove that if $x\geq 2$, then $\sum_{p\leq x}p^{-1}\leq \log\log x+5$.  It follows that
\begin{align*}
|g_{[\kd,\kd']}(s,\chi\otimes\bar{\chi}')|&\leq e^{e\omega([\kd,\kd'])}\exp\Big(ed^2\sum_{\N\kp\leq d^{2}}\frac{1}{\N\kp}\Big)\\
&\leq e^{e\omega([\kd,\kd'])}\exp\Big(e d^2 n_K \sum_{p\leq d^{2}}\frac{1}{p}\Big)\leq e^{e\omega([\kd,\kd'])}\exp(e d^{2}n_K(\Log\Log d+6)).
\end{align*}
Bounding $\omega([\kd,\kd'])$ using \cref{lem:prime_divisors} and our choice of $\delta$, we conclude the desired bound.  The other inequality is proved similarly.  One starts with the bounds
\[
\Big|\prod_{\kp\mid\kf_{\chi}\kf_{\chi'}}\prod_{j=1}^{\chi(1)}\prod_{j'=1}^{\chi'(1)}\frac{1-\alpha_{j,j',\chi\otimes\bar{\chi}'}(\kp)\N\kp^{-s}}{1-\alpha_{j,\chi}(\kp)\alpha_{j'\bar{\chi}'}(\kp)\N\kp^{-s}}\Big|\leq \prod_{\kp\mid\kf_{\chi}\kf_{\chi'}}\Big(1+\frac{2}{\N\kp^{1-\delta}-1}\Big)^{d^2}
\]
and
\[
\Big(1+\frac{2}{\N\kp^{1-\delta}-1}\Big)^{d^2}\leq \begin{cases}
   e^{2e}&\mbox{if $\N\kp>d^2$,}\\
    \exp(4ed^2/\N\kp)&\mbox{if $\N\kp\leq d^2$.}
\end{cases}\qedhere
\]
\end{proof}

\begin{lemma}
\label{lem:convexity}
If $\chi,\chi'\in\Irr(G)$ and $t\in\R$, then
\[
|L(1-\tfrac{1}{2\Log d}+it,\chi\otimes\bar{\chi}')|\ll (2/\delta)^{d^2n_K}(q(|t|+3)^{d^2n_K})^{\frac{\delta}{1+2\delta}}.
\]
\end{lemma}
\begin{proof}
The constants $\mu_{\chi\otimes\chi'}(j,j')$ are either $0$ or $1$, and the Frobenius eigenvalues $\alpha_{j,j',\chi\otimes\chi'}(\kp)$ all have modulus at most $1$.  Therefore, after minor adjustments to the proof of \cite[Theorem~1.1]{ST} (including shifting the terms $\mu_{\chi\otimes\bar{\chi}'}(j)$ in $L_{\infty}(s,\chi\otimes\bar{\chi}')$ by $it$ and applying the bound \cite[Lemma~1]{Rademacher}), we arrive at the bound
\begin{align*}
\log|L(\tfrac{1}{2}+it,\chi\otimes\bar{\chi}')|&\leq \frac{1}{4}\log(q(|t|+3)^{d^2 n_K})+2\log|L(\tfrac{3}{2}+it,\chi\otimes\bar{\chi}')|+O(1).
\end{align*}
For $w>0$, the trivial bound $|L(1+w+it,\chi\otimes\bar{\chi}')|\leq \zeta(1+w)^{d^2n_K}\le (2/w)^{d^2 n_K}$ follows from \eqref{eqn:eigenvalue definition of Artin $L$-functions}. The lemma now follows from the bound on $\re(s)=\frac{1}{2}$, the bound on $\re(s)=1+\delta$, and the Phragm{\'e}n--Lindel{\"o}f principle \cite[Lemma~5.53]{IK}.
\end{proof}
\begin{proposition}
\label{prop:local_density}
Let $x,T\geq 1$. If $\kd\in\cD_{\chi}(z)$ and $\kd'\in\cD_{\chi'}(z)$, then
\begin{align*}
&\sum_{[\kd,\kd']\mid \ka}\,\sum_{(\mu_{\kp})_{\kp}\in\underline{\mu}[\ka]}\Big(\prod_{\kp}s_{\mu_{\kp}}(A_{\chi}(\kp))\overline{s_{\mu_{\kp}}(A_{\chi'}(\kp))}\Big)\Phi_j\Big(T\log\frac{\N\ka}{x}\Big)\\
&=\frac{x}{T}\hat{\Phi}(1/T)\kappa_{\chi,\bar{\chi}'}g_{\chi}([\kd,\kd'])+O(x^{1-\delta}T^{d^2 n_K\delta}z^{2e\delta}\cQ^{213\delta}).
\end{align*}
\end{proposition}

\begin{proof}
This follows from \eqref{eqn:integranded}, \eqref{eqn:residued}, \eqref{eqn:residued2}, and Lemmata \ref{lem:smoothlargesieve}, \ref{lem:finite_Euler}, and \ref{lem:convexity}.
\end{proof}


\begin{proposition}
\label{prop:lower_logz}
There exists a constant $\Cr{lower00}>426$ such that
\[
\kappa_{\chi,\bar{\chi}}\Big(\sum_{\N\ka\leq z}\frac{1}{\N\ka}\sum_{(\mu_{\kp})_{\kp}\in\underline{\mu}[\ka]}\prod_{\kp}|s_{\mu_{\kp}}(A_{\chi}(\kp))|^2\Big)^{-1}\ll \frac{1}{\log z},\qquad z\geq \cQ^{\Cr{lower00}}.
\]
\end{proposition}
\begin{proof}
Note that
\[
\sum_{\N\ka\leq z}\frac{1}{\N\ka}\sum_{(\mu_{\kp})_{\kp}\in\underline{\mu}[\ka]}\prod_{\kp}|s_{\mu_{\kp}}(A_{\chi}(\kp))|^2\geq 1+\sum_{\sqrt{z}\leq \N\ka\leq z}\frac{1}{\N\ka}\sum_{(\mu_{\kp})_{\kp}\in\underline{\mu}[\ka]}\prod_{\kp}|s_{\mu_{\kp}}(A_{\chi}(\kp))|^2.
\]
Let $\Phi_2$ be as in \cref{lem:smoothlargesieve}.  We decompose $[\sqrt{z},z]$ into $O(\log z)$ dyadic blocks of the form $[y,ey]$.  It follows from \cref{prop:local_density} (with $T=1$, $\kd=\kd'=\cO_K$, $x$ replaced by $y$, and letting $z=1$ therein) that
\begin{align*}
\sum_{y\leq n\leq ey}\frac{1}{\N\ka}\sum_{(\mu_{\kp})_{\kp}\in\underline{\mu}[\ka]}\prod_{\kp}|s_{\mu_{\kp}}(A_{\chi}(\kp))|^2&\geq \frac{1}{ey}\sum_{y\leq \N\ka\leq ey}\sum_{(\mu_{\kp})_{\kp}\in\underline{\mu}[\ka]}\prod_{\kp}|s_{\mu_{\kp}}(A_{\chi}(\kp))|^2\\
&\geq \frac{1}{ey}\sum_{\ka}\Big[\sum_{(\mu_{\kp})_{\kp}\in\underline{\mu}[\ka]}\prod_{\kp}|s_{\mu_{\kp}}(A_{\chi}(\kp))|^2\Big]\Phi_2\Big(\log\frac{\N\ka}{y}\Big)\\
&\geq \frac{\hat{\Phi}_2(1)}{e}\kappa_{\chi,\bar{\chi}}+O(y^{-\delta} \cQ^{213\delta}).
\end{align*}

We sum the contribution from each dyadic block, thus obtaining
\[
\sum_{\N\ka\leq z}\frac{1}{\N\ka}\sum_{(\mu_{\kp})_{\kp}\in\underline{\mu}[\ka]}\prod_{\kp}|s_{\mu_{\kp}}(A_{\chi}(\kp))|^2\geq 1+\frac{\hat\Phi_2(1)}{2e}\kappa_{\chi,\bar{\chi}}\log z+O(z^{-\frac{\delta}{2}}\cQ^{213\delta}).
\]
Consequently, there exists a constant $\Cr{lower00}>426$ such that if $z\geq \cQ^{\Cr{lower00}}$, then
\[
\sum_{\N\ka\leq z}\frac{1}{\N\ka}\sum_{(\mu_{\kp})_{\kp}\in\underline{\mu}[\ka]}\prod_{\kp}|s_{\mu_{\kp}}(A_{\chi}(\kp))|^2\gg 1+\kappa_{\chi,\bar{\chi}}\log z.
\]
Since $\kappa_{\chi,\bar{\chi}}/(1+\kappa_{\chi,\bar{\chi}}\log z)\leq 1/\log z$ for $z>1$, the desired lower bound follows.
\end{proof}

\subsection{A ``pre-sifted'' large sieve inequality}

Let $b$ be a complex-valued function on the nonzero ideals of $\cO_K$, and let $\beta\colon\Irr(G)\to\mathbb{C}$ be a function.  Write $\beta_{\chi}=\beta(\chi)$.  Their $\ell^2$ norms $\|b\|_2$ and $\|\beta\|_2$ are defined by
\[
\|b\|_2 = \Big(\sum_{x<\N\ka\leq xe^{1/T}}|b(\ka)|^2\Big)^{1/2},\qquad \|\beta\|_2=\Big(\sum_{\chi\in\Irr(G)}|\beta_{\chi}|^2\Big)^{1/2}.
\]

Consider
\begin{equation}
\label{eqn:large_sieve_1}
\begin{aligned}
&\sup_{\|b\|_2=1}\sum_{\chi\in\Irr(G)}\Big|\sum_{\N\ka\in(x,xe^{1/T}]}\lambda_{\chi}(\ka)\Big(\sum_{\kd\mid\gcd(\ka,P_{\chi}(z))}\rho_{\chi}(\kd)\Big)b(\ka)\Big|^2\\
&=\sup_{\|b\|_2=1}\sum_{\chi\in\Irr(G)}\Big|\sum_{\N\ka\in(x,xe^{1/T}]}\Big[\prod_{\kp}s_{(\mathrm{ord}_{\kp}(\ka),0,0,\ldots)}(A_{\chi}(\kp))\Big]\Big(\sum_{\kd\mid\gcd(\ka,P_{\chi}(z))}\rho_{\chi}(\kd)\Big)b(\ka)\Big|^2.
\end{aligned}
\end{equation}
By the duality principle for bilinear forms \cite[Section 7.1]{IK}, \eqref{eqn:large_sieve_1} equals the supremum over all $\beta$ with $\|\beta\|_2=1$ of
\begin{equation}
\label{eqn:large_sieve_2}
\sum_{\N\ka\in(x,xe^{1/T}]}\Big|\sum_{\chi\in\Irr(G)}\Big[\prod_{\kp}s_{(\mathrm{ord}_{\kp}(\ka),0,0,\ldots)}(A_{\chi}(\kp))\Big]\Big(\sum_{\kd\mid\gcd(\ka,P_{\chi}(z))}\rho_{\chi}(\kd)\Big)\beta_{\chi}\Big|^2.
\end{equation}
Let $\Phi$ be a fixed nonnegative smooth function supported on a compact subset of $[-2,2]$ such that $\Phi(t)= 1$ for $t\in[0,1]$.  Since $(\mathrm{ord}_{\kp}(\ka),0,0,\ldots)\in\underline{\mu}[\ka]$, \eqref{eqn:large_sieve_2} is
\begin{equation}
\label{eqn:large_sieve_5}
\leq \sum_{\ka}\sum_{(\mu_{\kp})_{\kp}\in\underline{\mu}[\ka]}\Big|\sum_{\chi\in\Irr(G)}\Big[\prod_{\kp}s_{\mu_{\kp}}(A_{\chi}(\kp))\Big]\Big(\sum_{\kd\mid\gcd(\ka,P_{\chi}(z))}\rho_{\chi}(\kd)\Big)\beta_{\chi}\Big|^2 \Phi\Big(T\log\frac{\N\ka}{x}\Big).
\end{equation}

Write $[\kd_1,\kd_2]$ for the least common multiple of $\kd_1$ and $\kd_2$.  We expand the square and swap the order of summation, using the fact that $\rho_\chi(\kd)=0$ unless $\kd\in \cD_\chi(z)$.  Thus, \eqref{eqn:large_sieve_5} equals
\begin{equation}
\label{eqn:large_sieve_6}
\sum_{\chi,\chi'}\beta_{\chi}\overline{\beta_{\chi'}}\sum_{\substack{\kd\mid P_\chi(z)\\\kd'\mid P_{\chi'}(z)}}\rho_{\chi}(\kd)\rho_{\chi'}(\kd)\sum_{[\kd,\kd']\mid\ka}\Big[\sum_{(\mu_\kp)_\kp\in\underline{\mu}[\ka]}\prod_{\kp}s_{\mu_{\kp}}(A_{\chi}(\kp))s_{\mu_{\kp}}(A_{\bar{\chi}'}(\kp))\Big]\Phi\Big(T\log\frac{\N\ka}{x}\Big).\hspace{-4mm}
\end{equation}
By \cref{prop:local_density} and the fact that $\kappa_{\chi,\bar{\chi}'}\ne 0$ if and only if $\chi=\chi'$, \eqref{eqn:large_sieve_6} equals
\begin{equation}
\label{eqn:large_sieve_7}
\begin{aligned}
&\frac{x}{T}\hat{\Phi}(1/T)\sum_{\chi\in\Irr(G)}|\beta_{\chi}|^2 \kappa_{\chi,\bar{\chi}}\sum_{\kd,\kd' \in \cD_{\chi}(z)}\rho_{\chi}(\kd)\rho_{\chi}(\kd')g_{\chi}([\kd,\kd'])\\
&+O\Big(x^{1-\delta}T^{d^2 n_K\delta}z^{2e\delta}\cQ^{213\delta}\sum_{\chi,\chi'\in\Irr(G)}|\beta_{\chi}\beta_{\chi'}|\sum_{\substack{\kd \in \cD_{\chi}(z) \\ \kd'\in\cD_{\chi'}(z)}}1\Big).
\end{aligned}
\end{equation}
Weiss \cite[Proof of Lemma 1.12(a)]{Weiss} proved that
\begin{equation}
\label{eqn:Weiss_integers}
\sum_{\N\ka\leq z}1\leq (1+a)^{n_K}z^{1+\frac{1}{a}},\qquad a\geq 1,\quad z>1.
\end{equation}
By \eqref{eqn:Weiss_integers} with $a=2\Log d$, and the inequality of arithmetic and geometric means, and the fact that $\|\beta\|_2=1$, it follows that \eqref{eqn:large_sieve_6} equals
\begin{equation}
\label{eqn:large_sieve_8}
\begin{aligned}
&\frac{x}{T}\hat{\Phi}(1/T)\sum_{\chi\in\Irr(G)}|\beta_{\chi}|^2\kappa_{\chi,\bar{\chi}}\sum_{\kd,\kd' \in \cD_{\chi}(z)}\rho_{\chi}(\kd)\rho_{\chi}(\kd')g_{\chi}([\kd,\kd'])\\
&+O(x^{1-\delta}T^{d^2 n_K\delta}z^{8\delta+2}\cQ^{214\delta}|\Irr(G)|).
\end{aligned}
\end{equation}

Proceeding as in the formulation of Selberg's sieve in \cite[Theorem 7.1]{FI}, we find that for each $\chi\in\Irr(G)$, there exists a function $\rho_{\chi}(\kd)$ satisfying \eqref{eqn:Selberg_weights} such that
\begin{align*}
\sum_{\kd,\kd'\in\cD_{\chi}(z)}\rho_{\chi}(\kd)\rho_{\chi}(\kd')g_{\chi}([\kd,\kd'])=\sum_{\substack{\N\kd\leq z^2 \\ \kd\mid P_{\chi}(z)}}\prod_{\kp\mid\kd}\frac{g_{\chi}(\kp)}{1-g_{\chi}(\kp)}&\leq \Big(\sum_{\substack{\N\ka\leq z \\ \textup{$\ka$ squarefree}}}\prod_{\kp\mid\ka}\sum_{\mu}\frac{|s_{\mu}(A_{\chi}(\kp))|^2}{\N\kp^{|\mu|s}}\Big)^{-1}\\
&\leq \Big(\sum_{\N\ka\leq z}\frac{1}{\N\ka}\sum_{(\mu_{\kp})_{\kp}\in\underline{\mu}[\ka]}\prod_{\kp}|s_{\mu_{\kp}}(A_{\chi}(\kp))|^2\Big)^{-1}.
\end{align*}
If $z\geq \cQ^{\Cr{lower00}}$, then \cref{prop:lower_logz} (and $\|\beta\|_2=1$) shows that \eqref{eqn:large_sieve_8}, hence \eqref{eqn:large_sieve_1}, is
\begin{align*}
&\ll \frac{x}{T\log z}\sum_{\chi\in\Irr(G)}|\beta_{\chi}|^2+x^{1-\delta}T^{d^2 n_K\delta}z^{8\delta+2}\cQ^{214\delta}|\Irr(G)|\\
&\ll \frac{x}{T\log z}+x^{1-\delta}T^{d^2 n_K\delta}z^{8\delta+2}\cQ^{214\delta}|\Irr(G)|.
\end{align*}

Our bound for \eqref{eqn:large_sieve_1} can be rephrased as follows.  Let $b$ be any complex-valued function on the nonzero ideals of $\cO_K$.  Let $\rho_{\chi}$ satisfy \eqref{eqn:Selberg_weights}.  If $z\geq \cQ^{\Cr{lower00}}$, then
\begin{align*}
&\sum_{\chi\in\Irr(G)}\Big|\sum_{\N\ka\in(x,xe^{1/T}]}\lambda_{\chi}(\ka)\Big(\sum_{\kd\mid\gcd(\ka,P_{\chi}(z))}\rho_{\chi}(\kd)\Big)b(\ka)\Big|^2\\
&\ll \Big(\frac{x}{T\log z}+x^{1-\delta}T^{d^2 n_K\delta}z^{8\delta+2}\cQ^{214\delta}|\Irr(G)|\Big)\sum_{\N\ka\in(x,xe^{1/T}]}|b(\ka)|^2.
\end{align*}
We now choose $b(\ka)$ to be supported on the prime ideals with norm $\N\kp>z$.  Noting the identity $\lambda_{\chi}(\kp)=a_{\chi}(\kp)$, we find that if $z\geq \cQ^{\Cr{lower00}}$, then
\begin{equation}
\label{eqn:large_sieve_10}
\begin{aligned}
&\sum_{\chi\in\Irr(G)}\Big|\sum_{\substack{\N\kp\in(x,xe^{1/T}] \\ \N\kp>z}}a_{\chi}(\kp)b(\kp)\Big|^2\\
&\ll \Big(\frac{x}{T\log z}+x^{1-\delta}T^{d^2 n_K\delta}z^{8\delta+2}\cQ^{214\delta}|\Irr(G)|\Big)\sum_{\substack{\N\kp\in(x,xe^{1/T}] \\ \N\kp>z}}|b(\kp)|^2.
\end{aligned}
\end{equation}

\begin{proof}[Proof of \cref{prop:hybrid_MVT}] 
A formal generalization of \cite[Theorem 1]{Gallagher} to implies that if $c(\ka)$ is a complex-valued function on the nonzero ideals of $\cO_F$ such that $\sum_{\ka}|c(\ka)|<\infty$, then
\[
\int_{-T}^T\Big|\sum_{\ka}\frac{c(\ka)}{\N\ka^{it}}\Big|^2 dt\ll T^2\int_0^{\infty}\Big|\sum_{\N\ka\in(x,xe^{1/T}]}c(\ka)\Big|^2\frac{dx}{x}.
\]
It follows that
\begin{equation}
\label{eqn:MVT_proof_01}
\sum_{\chi\in\Irr(G)}\int_{-T}^{T}\Big|\sum_{\N\kp>z}\frac{a_{\chi}(\kp)b(\kp)}{\N\kp^{it}}\Big|^2 dt\ll T^2\int_0^{\infty}\sum_{\chi\in\Irr(G)}\Big|\sum_{\substack{\N\kp\in(x,xe^{1/T}] \\ \N\kp>z}}a_{\chi}(\kp)b(\kp)\Big|^2\frac{dx}{x}.
\end{equation}
We apply \eqref{eqn:large_sieve_10} so that \eqref{eqn:MVT_proof_01} is
\begin{align*}
	&\ll T^2\int_0^\infty \Big(\frac{x}{T\log z}+x^{1-\delta}T^{d^2 n_K\delta}z^{8\delta+2}\cQ^{214\delta}|\Irr(G)|\Big)\sum_{\substack{\N\kp\in(x,xe^{1/T}] \\ \N\kp>z}}|b(\kp)|^2 dx
	\\
	&\ll T^2\sum_{\N\kp>z}|b(\kp)|^2\int_{e^{-1/T}\N\kp}^{\N\kp}\Big(\frac{1}{T\log z}+x^{-\delta}T^{d^2 n_K\delta}z^{8\delta+2}\cQ^{214\delta}|\Irr(G)|\Big) dx
	\\
	&\ll\frac{1}{\log z}\sum_{\N\kp>z}|b(\kp)|^2 \N\kp\Big(1+\frac{\cQ^{214\delta}T^{2d^2 n_K\delta}z^{8\delta+3}|\Irr(G)|}{\N\kp^{\delta}}\Big).
\end{align*}

Since $\Cr{lower00}>426$, we choose $z = \cQ^{\Cr{lower00}} T^{2d^2n_K\delta}$ and $y\ge z^{10/\delta}$.  Imposing the constraint on $b(\kp)$ that $b(\kp)=0$ unless $\N\kp\geq y$, we conclude that
\begin{align*}
\sum_{\chi\in\Irr(G)}\int_{-T}^{T}\Big|\sum_{\N\kp>y}\frac{a_{\chi}(\kp)b(\kp)}{\N\kp^{it}}\Big|^2 dt\ll \frac{1}{\delta\log y}\sum_{\N\kp>y}|b(\kp)|^2 \N\kp.
\end{align*}
This establishes \cref{prop:hybrid_MVT}.
\end{proof}

\section{Proof of \cref{thm:LFZDE}}
\label{sec:LFZDE}

Let $T\ge 1$. From now on, we assume that $\eta$ is in the range
\begin{equation}\label{eqn:range of eta}
	\frac{\Cr{ZFR}}{\log(q T^{d^2n_K})}\leq \eta\leq\frac{1}{330},
\end{equation}
where $\Cr{ZFR}$ is as in \cref{thm:ZFR}.  We may assume without loss of generality that $\Cr{ZFR}<1/330$.

For $\tau\in\R$ such that $|\tau|\le T$, we write $s=1+\eta+i\tau$. Furthermore, for any integer $k\ge 2$ and $\chi\in\Irr(G)$, we write
\[
F(z) = \frac{L'}{L}(z,\chi)+\frac{L'}{L}(z+1-\beta_1,\chi\otimes\chi_1),\qquad G_k(z)=\frac{(-1)^k}{k!}F^{(k)}(z).
\]
We define
\[
\mathbf{1}(\tau)=\begin{cases}
1&\mbox{if $|\tau|<200\eta$,}\\
0&\mbox{if $|\tau|\geq 200\eta$.}
\end{cases}
\]

\begin{lemma}\label{lem:G_k approximation}
There exists a constant $\Cl[abcon]{turan1}>0$ such that if
\begin{equation}
\label{eqn:K_range}
\cK\ge 1600\eta\log(q T^{d^2n_K}) + \Cr{turan1}
\end{equation}
and $\eta$ satisfies \eqref{eqn:range of eta}, then there exists $k\in[\cK,2\cK]$ such that
\begin{equation}
\eta^{k+1}|G_k(s)|+\delta(\chi)\mathbf{1}(\tau)\min\{1,2^{k+4}((1-\beta_1)/\eta)^{1/2}\}\geq\frac{1}{2(100)^{k+1}}.
\label{eqn:G_lower}
\end{equation}
\end{lemma}

\begin{proof}
The proof proceeds as in \cite[(4.18)]{BTZ}, with Lemma 4.2 therein replaced by \cref{lem:bound of Dirichlet series}, and with Lemma 4.3 therein replaced by \cref{lem:Linnik}.  Unlike in \cite{BTZ}, the numbers $\mu_{\chi}(j)$ appearing in the local $L$-functions $L_v(s,\chi)$ at archimedean places $v$ satisfy $\re(\mu_{\chi}(j))\geq 0$, which ultimately permits us to assume an upper bound on $\eta$ that is independent of $d$.
\end{proof}

We now need an upper bound for $|G_k(s)|$. Define
\begin{equation}\label{eqn:N_1 and N_2}
	N_1 = \exp(\cK/(300\eta)),\qquad N_2 = \exp(40\cK/\eta) = N_1^{12000}.
\end{equation}

\begin{lemma}
\label{lem:upper_bound}
If $\cK$ satisfies \eqref{eqn:K_range}, then there exists $k\in[\mathcal{K},2\mathcal{K}]$ such that
\begin{align*}
\eta^{k+1}|G_k(s)|\leq \eta^2\int_{N_1}^{N_2}\Big|\sum_{N_1\leq \N\kp\leq u}\frac{a_{\chi}(\kp)(1+\chi_1(\kp)\N\kp^{\beta_{1}-1})\log\N\kp}{\N\kp^{1+i\tau}}\Big|\frac{du}{u}+O\Big(\frac{k}{110^k}\Big).
\end{align*}
\end{lemma}

\begin{proof}
    The proof proceeds as in \cite[(4.23)]{BTZ}, with Lemma 4.2 therein replaced by \cref{lem:bound of Dirichlet series}, and with Lemma 4.4 therein replaced by \eqref{eqn:trivial bound on a_chi}.  The biggest departure from the approach in \cite{BTZ} is how we handle the sum
\[
\eta\sum_{\substack{\N\ka\in[N_1,N_2] \\ \textup{$\ka$ composite}}} \frac{(a_{\chi}(\ka)+a_{\chi\otimes\chi_1}(\ka)\N\ka^{\beta_1-1})\Lambda_K(\ka)}{\N\ka^{1+i\tau}}j_k(\eta\log\N\ka),\qquad j_k(u)=\frac{u^k e^{-u}}{k!}.
\]
To handle this sum, we note that if $\ka$ is composite, then $\Lambda_K(\ka)$ is nonzero only when $\ka=\kp^r$ for some integer $r\ge 2$. It follows that the sum above is
    \begin{equation}
    \label{eqn:sum over composite ideals}
    \begin{aligned}
\ll d\eta\sum_{\kp}\sum_{r=2}^\infty\frac{\log\N\kp}{\N\kp^{r(1+\eta)}}\frac{(\eta\log\N\kp^r)}{k!}
    \end{aligned}
    \end{equation}
    For any prime ideal $\kp$ and any integer $r\ge 2$, we have
    \[
    \frac{1}{\N\kp^{(1+\eta)r}} \frac{(\eta \log \N\kp^r)^k}{k!}= \frac{1}{\N\kp^{(1+\eta)r}} \frac{(3 \eta)^k(\log \N\kp^{r / 3})^k}{k!} \ll\Big(\frac{3}{330}\Big)^k \frac{1}{\N\kp^{(2 / 3+\eta)r}},
    \]
    where we have used the inequality $(\log u)^k\leq k! u$ (valid for $k\geq 1$ and $u\geq 1$) and the assumption that $\eta\le 1/330$. It follows that \eqref{eqn:sum over composite ideals} is
    \[
        \ll \frac{d\eta}{110^k}\sum_{\kp}\sum_{r=2}^\infty\frac{\log\N\kp}{\N\kp^{(2/3+\eta)r}}
        \ll \frac{d\eta}{110^k}\frac{\zeta_K'}{\zeta_K}(4/3+2\eta)\ll \frac{dn_K\eta}{110^k}\ll \frac{\cK}{110^k}\ll  \frac{k}{110^k}.\qedhere
    \]
\end{proof}

\subsection{Proof of \cref{thm:LFZDE}}

We now combine \cref{lem:upper_bound} with \eqref{eqn:G_lower}, following \cite[Section 4.4]{BTZ}. If there exists a nontrivial zero $\rho_0$ of $L(s,\chi)$ such that $|1+i\tau-\rho|\leq\eta$ and $\mathcal{K}$ satisfies \eqref{eqn:K_range} with a sufficiently large constant $\Cr{turan1}$, then there exists $k\in[\cK,2\cK]$ such that
\begin{align*}
1&\leq 4(100)^{k+1}\eta^2\int_{N_1}^{N_2}\Big|\sum_{N_1\leq \N\kp\leq u}\frac{a_{\chi}(\kp)(1+\chi_1(\kp)\N\kp^{\beta_{1}-1})\log\N\kp}{\N\kp^{1+i\tau}}\Big|\frac{du}{u}\\
&\quad +4(100)^{k+1}\delta(\chi)\mathbf{1}(\tau)\min\{1,2^{k+4}((1-\beta_1)/\eta)^{1/2}\}.
\end{align*}
We square both sides, apply Cauchy--Schwarz, and obtain that
\begin{align*}
1&\ll (100)^{4\cK}\eta^4\Big(\int_{N_1}^{N_2}\frac{du}{u}\Big)\int_{N_1}^{N_2}\Big|\sum_{N_1\leq \N\kp\leq u}\frac{a_{\chi}(\kp)(1+\chi_1(\kp)\N\kp^{\beta_{1}-1})\log\N\kp}{\N\kp^{1+i\tau}}\Big|^2\frac{du}{u}\\
&\quad +(100)^{4\cK}\delta(\chi)\mathbf{1}(\tau)\min\{1,2^{4\cK}(1-\beta_1)/\eta\}.
\end{align*}
Since $\int_{N_1}^{N_2}\frac{du}{u}\ll \frac{\cK}{\eta}$, we
we conclude that
\begin{align*}
1&\ll 200^{4\cK}\Big(\eta^3 \int_{N_1}^{N_2}\Big|\sum_{N_1\leq \N\kp\leq u}\frac{a_{\chi}(\kp)(1+\chi_1(\kp)\N\kp^{\beta_1-1})\log\N\kp}{\N\kp^{1+i\tau}}\Big|^2\frac{du}{u}+\delta(\chi)\mathbf{1}(\tau)\min\Big\{1,\frac{1-\beta_1}{\eta}\Big\}\Big).
\end{align*}
Now, \cref{lem:Linnik} implies that
\begin{align*}
\textstyle\#\{\rho=\beta+i\gamma\neq\beta_1\colon \Lambda(\rho,\chi)=0,~\beta\geq 1-\frac{\eta}{2},~|\gamma-\tau|\leq\frac{\eta}{2}\}\ll{\eta\log(q T^{d^2n_K})}\ll \cK.
\end{align*}
Combining the two preceding displayed equations, we have
\begin{align*}
&\#\{\rho=\beta+i\gamma\neq\beta_1\colon \beta\geq 1-\tfrac{\eta}{2},~|\gamma-\tau|\leq\tfrac{\eta}{2}\}\\
&\ll 200^{4\cK}\cK\eta^3 \int_{N_1}^{N_2}\Big|\sum_{N_1\leq \N\kp\leq u}\frac{a_{\chi}(\kp)(1+\chi_1(\kp)\N\kp^{\beta_1-1})\log\N\kp}{\N\kp^{1+i\tau}}\Big|^2\frac{du}{u}\\
&+200^{4\cK}\cK\delta(\chi)\mathbf{1}(\tau)\min\Big\{1,\frac{1-\beta_1}{\eta}\Big\}. 
\end{align*}
Since $1/\eta\ll \log(q T^{d^2n_K})$, we see that $\min\{1,(1-\beta_1)/\eta\}\ll \nu(q T^{d^2n_K})$. Integrating the above equation over $|\tau|\leq T$ and summing over $\chi\in\Irr(G)$, we obtain
\begin{equation}\label{eqn:ready for large sieve}
\begin{aligned}
&\sum_{\chi\in\Irr(G)}N_{\chi}^*(1-\tfrac{\eta}{2},T)\ll 200^{4\cK}\cK\nu(q T^{d^2n_K})\\
&+200^{4\cK}\cK \eta^2 \int_{N_1}^{N_2}\Big[\sum_{\chi\in\Irr(G)}\int_{-T}^{T}\Big|\sum_{N_1\leq \N\kp\leq u}\frac{a_{\chi}(\kp)(1+\chi_1(\kp)\N\kp^{\beta_1-1})\log\N\kp}{\N\kp^{1+i\tau}}\Big|^2 d\tau\Big]\frac{du}{u}.
\end{aligned}
\end{equation}

We now apply \cref{prop:hybrid_MVT}. Recall our assumption on $\cK$ in \eqref{eqn:K_range} and our definitions of $N_1$ and $N_2$ in \eqref{eqn:N_1 and N_2}. We choose
\begin{equation}
\label{eqn:KK_def}
\begin{gathered}
\cK=1600\eta\log(\cQ^{\Cr{lower00}\Log d}T^{d^2n_K})+\Cr{turan1} ,\qquad y=N_1,\qquad u\in[N_1,N_2],\\
b(\kp)=\begin{cases}
(1+\chi_1(\kp)\N\kp^{\beta_1-1})\frac{\log\N\kp}{\N\kp}&\mbox{if $\N\kp\in[N_1,u]$,}\\
0&\mbox{otherwise.}
\end{cases}
\end{gathered}
\end{equation}
\cref{prop:hybrid_MVT} implies that
\begin{align*}
&\sum_{\chi\in\Irr(G)}\int_{-T}^{T}\Big|\sum_{N_1\leq \N\kp\leq u}\frac{a_{\chi}(\kp)(1+\chi_1(\kp)\N\kp^{\beta_1-1})\log\N\kp}{\N\kp^{1+i\tau}}\Big|^2 d\tau\\
&\ll (\Log d)\sum_{N_1\leq \N\kp\leq u}\frac{(1+\chi_1(\kp)\N\kp^{\beta_1-1})^2 \log\N\kp}{\N\kp}.
\end{align*}
Weiss \cite[pp. 83-84]{Weiss} proved that
\[
\sum_{N_1\leq \N\kp\leq N_2}\frac{(1+\chi_1(\kp)\N\kp^{\beta_1-1})^2 \log\N\kp}{\N\kp}\ll \frac{\cK^2}{\eta}\min\Big\{1,\frac{1-\beta_1}{\eta}\Big\}\ll \frac{\cK^2}{\eta}\nu(q T^{d^2n_K}).
\]
It follows that \eqref{eqn:ready for large sieve} is
\begin{align*}
\ll 200^{4\cK}\cK\Big(\eta^2 (\Log d)\int_{N_1}^{N_2}\Big(\frac{\cK^2}{\eta}\nu(q T^{d^2 n_K})\Big)\frac{du}{u}+\nu(q T^{d^2n_K})\Big)
\ll 201^{4\cK}(\Log d)\nu(q T^{d^2 n_K}).
\end{align*}
Writing $\sigma = 1-\frac{\eta}{2}$ and recalling our definition of $\cK$ in \eqref{eqn:KK_def}, we conclude that if
\[
1-\frac{1}{660}\leq \sigma\leq 1-\frac{\Cr{ZFR}}{2\log(q T^{d^2n_K})},
\]
then
\begin{equation}
	\label{eqn:before Riemann con Mangoldt}
	\sum_{\chi\in\Irr(G)}N_{\chi}^*(\sigma,T)\ll (\Log d)(\cQ T^{d^2n_K})^{68\,000(\Log d)\Cr{lower00}(1-\sigma)}\nu(q T^{d^2n_K}).
\end{equation}

In the region $\sigma\geq 1-\Cr{ZFR}/(2\log(q T^{d^2n_K}))$, the product $\prod_{\chi\in\Irr(G)}L(\sigma+it,\chi)$ has at most one zero per \cref{thm:ZFR}.  Thus, \eqref{eqn:before Riemann con Mangoldt} holds for all $\sigma\geq 1-\frac{1}{660}$.  In light of the trivial bound
\[
\sum_{\chi\in\Irr(G)}N_{\chi}(0,T)\ll |\Irr(G)|T\log(q_{L/K}T^{dn_K})
\]
that follows from \cref{lem:Riemann-vonMangoldt}, we conclude that there exists a constant $\Cl[abcon]{LFZDE_end}>0$ such that
\[
\sum_{\chi\in\Irr(G)}N_{\chi}^*(\sigma,T)\ll (\Log d)(\cQ T^{d^2n_K})^{\Cr{LFZDE_end}(\Log d)(1-\sigma)}\nu(q T^{d^2n_K}),\qquad\sigma\geq 0.
\]

\bibliographystyle{abbrv}
\bibliography{ThornerZhang_Chebotarev}
\end{document}